\documentclass[11pt, a4paper, oneside, DIV11, final, pagebackref]{amsart}
\usepackage{amssymb}
\usepackage[latin1]{inputenc}
\usepackage[T1]{fontenc}
\usepackage[stretch=10]{microtype}
\usepackage{fixltx2e}
\usepackage{fix-cm}
\usepackage{mathtools}
\usepackage{tikz}
\usetikzlibrary{calc}
\usepackage[DIV11, headinclude=true]{typearea}
\usepackage{hyperref}

\providecommand{\href}[2]{#2}

\providecommand*{\backref}{}
\providecommand*{\backrefalt}{}
\renewcommand*{\backref}[1]{}
\renewcommand*{\backrefalt}[4]{%
	\ifcase #1 %
	\or
	  Cited page~#2.
	\else
	  Cited pages~#2.
	\fi
}

\newcommand{\ic}{\mathbf{i}}
\newcommand{\boA}{\mathcal{A}}
\newcommand{\boC}{\mathcal{C}}
\newcommand{\boH}{\mathcal{H}}
\newcommand{\boL}{\mathcal{L}}

\newcommand{\HH}{\mathbb{H}}
\newcommand{\hol}{\beta}
\newcommand{\N}{\mathbb{N}}
\newcommand{\Z}{\mathbb{Z}}
\newcommand{\R}{\mathbb{R}}

\newcommand{\Sbb}{\mathbb{S}}
\newcommand{\Pbb}{\mathbb{P}}
\newcommand{\dd}{\mathop{}\!\mathrm{d}}
\newcommand{\norm}[1]{\left\| #1 \right\|}
\newcommand{\st}{\::\:}
\newcommand{\SL} {{\mathrm{PSL}(2,\R)}}
\newcommand{\bartheta}{\overline{\theta}}
\newcommand{\bSigma}{\overline{\Sigma}}
\newcommand{\abs}[1]{\left|#1\right|}
\newcommand{\lgth}{\abs}

\renewcommand{\epsilon}{\varepsilon}

\renewcommand{\phi}{\varphi}
\renewcommand{\leq}{\leqslant}
\renewcommand{\geq}{\geqslant}

\DeclareMathOperator{\Card}{Card}
\DeclareMathOperator{\Leb}{Leb}
\DeclareMathOperator{\Press}{Pr}
\DeclareMathOperator{\dLeb}{dLeb}
\newcommand{\coloneqq}{\mathrel{\mathop:}=}

\newtheorem{thm}{Theorem}[section]
\newtheorem{prop}[thm]{Proposition}
\newtheorem{definition}[thm]{Definition}
\newtheorem{lem}[thm]{Lemma}
\newtheorem{cor}[thm]{Corollary}
\newtheorem*{prop*}{Proposition}

\theoremstyle{definition}

\newtheorem{rmk}[thm]{Remark}

\numberwithin{equation}{section}

\title[Local limit theorem in hyperbolic groups]
{Local limit theorem for symmetric random walks in Gromov-hyperbolic groups}
\author{S\'ebastien Gou\"ezel}

\address{IRMAR, CNRS UMR 6625,
Universit\'e de Rennes 1, 35042 Rennes, France}
\email{sebastien.gouezel@univ-rennes1.fr}
\date{\today}

\begin{document}

\begin{abstract}
Completing a strategy of Gou\"ezel and Lalley \cite{gouezel_lalley}, we
prove a local limit theorem for the random walk generated by any
symmetric finitely supported probability measure on a non-elementary
Gromov-hyperbolic group: denoting by $R$ the inverse of the spectral
radius of the random walk, the probability to return to the identity
at time $n$ behaves like $C R^{-n}n^{-3/2}$. An important step in the
proof is to extend Ancona's results on the Martin boundary up to the
spectral radius: we show that the Martin boundary for $R$-harmonic
functions coincides with the geometric boundary of the group. In an
appendix, we explain how the symmetry assumption of the measure can
be dispensed with for surface groups.
\end{abstract}

\maketitle

\section{Introduction}

Consider a countable group $\Gamma$ (with identity denoted by $e$),
together with a probability measure $\mu$ whose support generates
$\Gamma$ as a semigroup (we say that $\mu$ is admissible).
Multiplying random elements of $\Gamma$ distributed independently
according to $\mu$, one obtains a random walk on $\Gamma$. The local
limit problem consists in determining good asymptotics for the
transition probabilities $p_n(x,y)$ of this random walk. Let us
assume for simplicity that $\mu$ is finitely supported. For
$\Gamma=\Z^d$, simple Fourier computations show that $p_n(e,e)\sim C
n^{-d/2}$ if the walk is centered, and $p_n(e,e)\sim C R^{-n}
n^{-d/2}$ for some $R>1$ if the walk is not centered. Similar
asymptotics hold in nilpotent groups by the deep results of
Varopoulos and Alexopoulos \cite{alexopoulos}.

When the group is not amenable, $p_n(e,e)$ decays exponentially fast.
The situation is well understood for semisimple Lie groups and
absolutely continuous measures since the work of Bougerol
\cite{bougerol}: the probability to return to a fixed neighborhood of
the identity behaves like $C R^{-n} n^{-a}$ for some $R>1$ (depending
on the measure one considers) and some $a>1$ only depending on the
geometry of the group. In the simplest case of rank one groups,
$a=3/2$. It is reasonable to conjecture that similar asymptotics
(with the same $a$) hold for random walks on cocompact lattices of
such semisimple Lie groups, but the proofs of Bougerol (based on
representation theory) do not adapt well, and this question is
essentially open.

A notable exception is the case of free groups: in this situation,
the generating function of the transition probabilities (also called
the Green function) $G_r(x,y) = \sum r^n p_n(x,y)$ is an algebraic
function of $r$. A careful study of its first singularity then yields
the asymptotics of $p_n(x,y)$. For free groups, this is due to Lalley
\cite{lalley_freegroup}, and the asymptotics is of the form $p_n(x,y)
\sim C(x,y) R^{-n} n^{-3/2}$, in accordance with the results of
Bougerol in rank one Lie groups. Most free products can also be
treated similarly, see \cite[Chapter III]{woess} and references
therein.

Recently, together with Lalley, we were able to treat in
\cite{gouezel_lalley} some non-amenable groups where the Green
function is not expected to be algebraic. We proved that, for a
cocompact lattice of $\SL$, and for a finitely supported symmetric
measure $\mu$, the above asymptotics $p_n(x,y) \sim C(x,y)
R^{-n}n^{-3/2}$ still holds. Henceforth, we will refer to this
situation as the $\SL$-case. The overall strategy can in fact be
formulated in any Gromov-hyperbolic group (including in particular
all cocompact lattices in rank one semisimple Lie groups), but a
crucial point in the proof really relies on two-dimensional geometry.
In this article, we provide a completely different argument for this
crucial point, making it possible to extend the results of
\cite{gouezel_lalley} to any Gromov-hyperbolic group.

We say that the walk is \emph{aperiodic} if there exists an odd
integer $n$ such that $p_n(e,e)>0$. In this case, $p_n(e,e)>0$ for
all large enough $n$.

\begin{thm}
\label{main_thm}
Let $\Gamma$ be a finitely generated non-elementary Gromov-hyperbolic
group. Let $\mu$ be an admissible finitely supported symmetric
probability measure on $\Gamma$. Denote by $R>1$ the inverse of the
spectral radius of the corresponding random walk. For any $x, y\in
\Gamma$, there exists $C(x,y)>0$ such that
  \begin{equation*}
  p_n(x,y) \sim C(x,y)R^{-n} n^{-3/2}
  \end{equation*}
if the walk is aperiodic. If the walk is periodic, this asymptotics
holds for even (resp.\ odd) $n$ if the distance from $x$ to $y$ is
even (resp.\ odd).
\end{thm}

The proof of the analogous theorem in the $\SL$-case in
\cite{gouezel_lalley} is divided in three steps, as follows:
\begin{enumerate}
\item One shows that Ancona's results \cite{ancona} on the Martin
    boundary extend up to $r=R$. In particular, the Martin kernel
    $K_{r,\xi}(x)=G_r(x,\xi)/G_r(e,\xi)$ converges when $\xi$
    tends to a point in the geometric boundary of $\Gamma$,
    uniformly in $r\in [1,R]$.
\item Using the Cannon automaton coding geodesics in the group,
    and thermodynamic formalism in the resulting subshift of
    finite type, one gets estimates for the sums $\sum_{x\in
    \Gamma} G_r(e,x)G_r(x,e)$ when $r\to R$ in terms of a
    pressure function. This implies that $r\mapsto G_r(e,e)$
    almost satisfies a differential equation. Asymptotics of this
    function follow.
\item From the asymptotics of $G_r(e,e)$, one deduces the
    asymptotics of $p_n(e,e)$ using tauberian theorems (and a
    little bit of spectral theory). The asymptotics of $p_n(x,y)$
    are proved in the same way.
\end{enumerate}

From this point on, this article is subdivided into three sections,
each devoted to one of those three steps. We will give further
comments, explain quickly the arguments in \cite{gouezel_lalley}, and
insist on the differences between the $\SL$-case and the general case
of Gromov-hyperbolic groups. The main difference is in the first
step: the proof of \cite{gouezel_lalley} is deeply $2$-dimensional,
and the general argument is completely different. For the second
step, a significant technical complication appears: in the
$\SL$-case, the Cannon automaton (a combinatorial object coding the
geodesics in the group) is transitive, while this is not the case in
general. To overcome this difficulty, we use additional information
from the first step, and a technique of Calegari and Fujiwara
\cite{calegari_fujiwara}. Finally, the third step is exactly the same
in the $\SL$-case or in the general case, we will only give some
details for the convenience of the reader.

While we have tried to make this article as self contained as
possible, \cite{gouezel_lalley} provides a good introduction to some
concepts and techniques that we use. The letter $C$ denotes a
constant that may vary from line to line. Since most arguments work
exactly in the same way for symmetric or nonsymmetric measures, we
have written most proofs without using the assumption of symmetry. It
only plays a role in the proof of Lemma~\ref{lem_avoidballs} (the
central lemma to obtain Ancona inequalities) and in
Section~\ref{sec_tauberian}. We expect that the first step (Ancona
inequalities) should be true without any symmetry assumption on the
measure. While we are not able to prove it in general, we are able to
obtain it for cocompact discrete subgroups of $\SL$. The
identification of the Martin boundary at the spectral radius follows.
The argument is given in Appendix~\ref{sec_appendix}.

\section{Ancona inequalities up to the spectral radius}

\subsection{The Green function}
\label{subsec_green}

Consider an admissible finitely supported probability measure $\mu$
on a countable group $\Gamma$. It defines a random walk on $\Gamma$.
Let $R=R(\mu) = \limsup p_n(e,e)^{-1/n}$ (when $\mu$ is symmetric,
this is the inverse of the spectral radius of the Markov operator
associated to the random walk on $\ell^2$). The Green function is
defined for $1\leq r< R$ and $x,y\in \Gamma$ by $G_r(x,y)=\sum r^n
p_n(x,y)$. By a result of Guivarc'h, it is convergent even for $r=R$
if the group carries no recurrent random walk (this is in particular
true for non-amenable groups). One should think of $G_r(x,y)$ as the
average number of passages in $y$ if the random walk starts from $x$,
but for the measure $r\mu$ instead of $\mu$. In particular, for
larger $r$, $G_r$ gives more weight to longer paths.

If $\gamma=(x, x_1,\dotsc, x_{n-1}, y)$ is a path of length $n$ from
$x$ to $y$, its $r$-weight $w_r(\gamma)$ is $r^n \prod_{i=0}^{n-1}
p(x_i, x_{i+1})$ (where $x_0=x$ and $x_n=y$ by convention, and we
write $p(a,b)=\mu(a^{-1}b)$ for the probability to jump from $a$ to
$b$). By definition, $G_r(x,y)=\sum w_r(\gamma)$, where the sum is
over all paths from $x$ to $y$.

If $\Omega$ is a subset of $\Gamma$, one defines the restricted Green
function $G_r(x,y;\Omega)$ as $\sum w_r(\gamma)$ where the sum is
over all paths $\gamma=(x,x_1,\dotsc, x_{n-1},y)$ such that $x_i\in
\Omega$ for $1\leq i\leq n-1$. If $A$ is a subset of $\Gamma$ such
that any trajectory of the random walk from $x$ to $y$ has to go
through $A$, one has
  \begin{equation}
  \label{eq_Grcompose}
  G_r(x,y) = \sum_{a\in A} G_r(x,a; A^c) G_r(a,y) = \sum_{a\in A} G_r(x,a) G_r(a, y; A^c),
  \end{equation}
where $A^c$ denotes the complement of $A$. Indeed, the first (resp.\
second) formula is proved by splitting a path from $x$ to $y$
according to its first (resp.\ last) visit to $A$. More generally, if
$\Omega$ is a subset of $\Gamma$ containing $x$ and $y$, the above
formula holds restricted to $\Omega$, i.e.,
  \begin{equation*}
  G_r(x,y; \Omega) = \sum_{a\in A\cap \Omega} G_r(x,a; A^c\cap \Omega) G_r(a,y; \Omega)
  = \sum_{a\in A\cap \Omega} G_r(x,a;\Omega) G_r(a, y; A^c\cap \Omega).
  \end{equation*}

Assuming that $\Gamma$ is finitely generated, we can consider a word
distance $d$ on $\Gamma$ coming from a finite symmetric generating
set. If $x$ and $y$ are at distance $d$, there is a path from $x$ to
$y$ with probability bounded from below by $C^{-d}$, and staying
close to a geodesic segment from $x$ to $y$. We deduce that, for any
$z$,
  \begin{equation}
  \label{eq_harnack}
  C^{-d(x,y)} \leq G_r(x,z)/G_r(y,z) \leq C^{d(x,y)},
  \end{equation}
and similar inequalities hold for the Green functions restricted to
any set containing a fixed size neighborhood of a geodesic segment
from $x$ to $y$. These inequalities are called Harnack inequalities.

The first visit Green function is $F_r(x,y)=G_r(x,y; \{y\}^c)$. It
only takes into account the first visits to $y$. For $r=1$, this is
the probability to reach $y$ starting from $x$. One has
$G_r(x,y)=F_r(x,y) G_r(y,y)=F_r(x,y) G_r(e,e)$, by the
formula~\eqref{eq_Grcompose} for $A=\{y\}$. Moreover, $F_r(x,y)
G_r(y,z) \leq G_r(x,z)$ (since the concatenation of a path from $x$
to $y$ with a path from $y$ to $z$ gives a path from $x$ to $z$).
Dividing by $G_r(e,e)$, one gets
  \begin{equation}
  \label{eq_Fr_subadd}
  F_r(x,y)F_r(y,z) \leq F_r(x,z).
  \end{equation}
We also obtain
  \begin{equation}
  \label{eq_trivial_ancona}
  G_r(x,y)G_r(y,z) \leq G_R(e,e) G_r(x,z).
  \end{equation}

The Martin boundary is the set of pointwise limits of sequences of
functions $K_{r,y_n}(x)=G_r(x,y_n)/G_r(e,y_n)$ when $y_n$ tends to
infinity. Since these functions are normalized by $K_{r,y}(e)=1$, and
$r$-harmonic except at $y$, limits exist, are nonzero, and
$r$-harmonic everywhere (since the measure $\mu$ has finite support).
Understanding the Martin boundary amounts to understanding for which
sequences $y_n$ the functions $K_{r,y_n}$ converge.

The derivative of $G_r$ with respect to $r$ can be computed. Indeed,
  \begin{equation*}
  \label{eq_derive}
  (r G_r(x,y))'=\sum_{z\in \Gamma} G_r(x,z)G_r(z,y),
  \end{equation*}
where the prime indicates the derivative with respect to $r$. This
equation for $x=y=e$ shows that
  \begin{equation}
  \label{eq_eta}
  \eta(r) \coloneqq \sum_z G_r(e,z) G_r(z,e) < +\infty \quad \text{for all }r<R.
  \end{equation}

Let us introduce a convenient notation: we shall write
  \begin{equation}
  \label{eq_defHr}
  H_r(x,y)=G_r(x,y)G_r(y,x).
  \end{equation}
In the symmetric case, this is simply the square of the Green
function. With this notation, $\eta(r) = \sum_{z\in \Gamma}
H_r(e,z)$. Since $G_r$ satisfies~\eqref{eq_trivial_ancona}, $H_r$
also satisfies this inequality.

\subsection{Ancona inequalities}

Consider a finitely generated group $\Gamma$. Its Cayley graph is
endowed with the word metric coming from any finite set of
generators. One says that $\Gamma$ is Gromov-hyperbolic (or simply
hyperbolic) if there exists $\delta$ such that any geodesic triangle
in this Cayley graph is $\delta$-thin, i.e., each side of the
triangle is contained in the $\delta$-neighborhood of the union of
the two other sides. This notion is invariant under quasi-isometry,
and therefore independent of the choice of the generators (see
\cite{ghys_hyperbolique} for more details on hyperbolic groups). The
geometric intuition to have is that any finite set of points in an
hyperbolic group is isometric to a finite set of points in a tree, up
to some constant only depending on the number of points. In
particular, statements regarding the relative positions of points can
be reduced to statements in trees, that are easy to check
combinatorially. This intuition is made precise by the following
theorem (\cite[Theorem 2.12]{ghys_hyperbolique}).
\begin{thm}
\label{thm:tree_approx}
For any $n\in \N$ and $\delta>0$, there exists a constant
$C=C(n,\delta)$ with the following property. Consider a subset $A$ of
a $\delta$-hyperbolic space of cardinality at most $n$. There exists
a map $\Phi$ from $F$ to a metric tree such that, for any $x,y\in A$,
  \begin{equation*}
  d(x,y)-C \leq d(\Phi(x), \Phi(y)) \leq d(x,y).
  \end{equation*}
\end{thm}

An hyperbolic group $\Gamma$ (or more generally any geodesic
Gromov-hyperbolic space) has a well defined geometric boundary
$\partial \Gamma$: this is the set of semi-infinite geodesics, where
two such geodesics are identified if they stay a bounded distance
away. This boundary is a compact space, and $\Gamma\cup \partial
\Gamma$ is also compact.

Consider now an admissible finitely supported probability measure
$\mu$ on a non-elementary hyperbolic group $\Gamma$ (i.e., not
quasi-isometric to $\{0\}$ or $\Z$). Let $R = R(\mu)$, it is strictly
larger than $1$ since $\Gamma$ is not amenable. Ancona proved in
\cite{ancona} that, for any $r<R$, the Martin boundary for
$r$-harmonic functions coincides with the geometric boundary:
$K_{r,y_n}$ converges pointwise if and only if $y_n$ converges to a
point $\xi\in\partial \Gamma$, and the limits are different for
different points of the boundary.

A crucial inequality in Ancona's proof is the fact that the converse
inequality to~\eqref{eq_trivial_ancona} holds for any $r<R$ whenever
$y$ is close to a geodesic from $x$ to $z$ (with a constant a priori
depending on $r$). In other words, typical trajectories from $x$ to
$z$ follow the geodesic sufficiently well so that they are likely to
pass close to $y$. When $r$ increases, $G_r$ gives more and more
weight to long trajectories, that are more likely to go further from
the geodesic. Hence, this Ancona estimate is more and more subtle
when $r$ increases. Proving such an estimate for $r=R$ is a crucial
step in the proof of Theorem~\ref{main_thm}.

\begin{definition}
A probability measure $\mu$ on a Gromov-hyperbolic group $\Gamma$
satisfies uniform Ancona inequalities if there exists a constant
$C>0$ such that, for any $x,z\in \Gamma$ and for any $y$ close to a
geodesic segment from $x$ to $z$, for any $r\in [1,R(\mu)]$,
  \begin{equation*}
  G_r(x,z) \leq C G_r(x,y)G_r(y,z).
  \end{equation*}
\end{definition}

\begin{thm}
\label{thm_ancona}
If $\mu$ is admissible, finitely supported and symmetric on a
non-elementary Gromov-hyperbolic group, it satisfies uniform Ancona
inequalities.
\end{thm}

This result has been proved in \cite{gouezel_lalley} for cocompact
lattices of $\SL$, using very specific two-dimensional arguments. The
main idea in the new argument to follow is to combine a
supermultiplicativity estimate (originating in
\cite{peigne_supermultiplicative} for counting problems) with a
geometric construction of random barriers in hyperbolic space. We
will write $\lgth{x}$ for the distance of $x$ to the identity $e$
(for some fixed word distance), and $\Sbb_k$ for the sphere of radius
$k$ around $e$. The rest of this section is devoted to the proof of
Theorem~\ref{thm_ancona}. We fix a non-elementary Gromov-hyperbolic
group $\Gamma$ and an admissible probability measure $\mu$. We do not
assume yet that $\mu$ is symmetric, since it will only be important
in Lemma~\ref{lem_avoidballs} below.

\begin{lem}
\label{lem_rallonge}
There exists $C>0$ such that, for any $x,y\in \Gamma$, there exists
$a\in \Gamma$ of length at most $C$ such that $\lgth{xay}\geq
\lgth{x}+\lgth{y}$.
\end{lem}
\begin{proof}
Fix $C_0=C(4,\delta)>0$ such that any configuration of at most 4
points can be approximated by a tree with error at most $C_0$, as in
Theorem~\ref{thm:tree_approx}.

We will rely on the classical construction of free groups with two
generators in $\Gamma$ as follows. An hyperbolic element of $\Gamma$
is an element $u$ of $\Gamma$ such that the left-multiplication by
$u$ has two fixed points at infinity, an attracting one and a
repelling one, denoted by $u_+$ and $u_-$. Consider two hyperbolic
elements $u,v$ in $\Gamma$ such that the four points
$u_+,u_-,v_+,v_-$ are distinct (this is possible since $\Gamma$ is
non-elementary, see the proof of \cite[Theorem
8.37]{ghys_hyperbolique}), and fix small disjoint neighborhoods
$V(u_+),V(u_-),V(v_+),V(v_-)$ of those points in $\Gamma\cup\partial
\Gamma$. If $N$ is large enough, any $x$ in the complement of
$V(a_+)$ (for $a\in F=\{u, u^{-1}, v, v^{-1}\}$) shares only a short
beginning with $a^N$. More precisely, there exists $K>0$ independent
of $N$ such that, in a tree approximation $\Phi$ of $e,x,a^N$ with
error at most $C_0$, the branches from $\Phi(e)$ leading to $\Phi(x)$
and $\Phi(a^N)$ split before time $K$. Increasing $N$, we can also
assume that $\lgth{a^N}\geq 4K+3C_0$ for all $a\in F$.

\begin{figure}[htb]
\centering
  \begin{tikzpicture}[scale = 0.7]
  \draw[thick]
       (0,0) coordinate[label=left: {$\Phi(e)$}](e)
    -- (3,0) coordinate (first_split)
    -- ++(90:2) coordinate[label=right:{$\Phi(x)$}](x)
    (first_split)
    -- ++(-20:4) coordinate(second_split)
    -- ++( 25:3) coordinate[label=right:{$\Phi(xa^N y)$}](xaNy)
    (second_split)
    -- ++(-90:2) coordinate[label=right:{$\Phi(xa^N)$}]  (xaN);
  \draw [<->, thin]
     ($(first_split) + (-0.2,0)$) -- node[left] {$\leq K$} ($(x) + (-0.2,0)$);
  \draw [<->, thin]
     ($(second_split) + (-0.2,0)$) -- node[left]       {$\leq K$} ($(xaN) + (-0.2,0)$);
  \draw [<->, thin]
     ($(first_split) + (70:0.2)$)  -- node[above right] {$\geq 2K+2C_0$} ($(second_split) + (70:0.2)$);
  \end{tikzpicture}
\caption{Approximating tree for $e, x, xa^N, xa^Ny$.}
\label{approx_tree}
\end{figure}
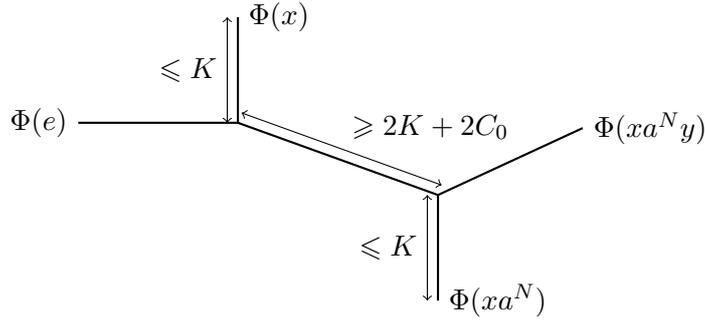

Consider now two points $x,y\in \Gamma$. One can choose $a\in F$ such
that $x^{-1}\not\in V(a_+)$ and $y\not\in V(a_-)$. Consider a tree
approximation $\Phi : \{e, x, xa^N, xa^N y\}\to T$. The geodesic
paths from $\Phi(x)$ to, respectively, $\Phi(e)$ and $\Phi(xa^N)$,
split before time $K$ by construction since $x^{-1}\not \in V(a_+)$.
In the same way, the paths from $\Phi(xa^N)$ to, respectively,
$\Phi(x)$ and $\Phi(xa^N y)$ also split before time $K$ since $y\not
\in V(a_-)$ (see~Figure~\ref{approx_tree}). Hence,
  \begin{align*}
  \lgth{xa^N y} &
  \geq
  d(\Phi(e),\Phi(xa^Ny))
  \\&\geq d(\Phi(e), \Phi(x)) + d(\Phi(x), \Phi(xa^N))
      + d(\Phi(xa^N), \Phi(xa^N y)) -  2 \cdot 2K
  \\&
  \geq \lgth{x}-C_0 + \lgth{a^N}-C_0 + \lgth{y} -C_0 -4K
  \\&
  \geq \lgth{x} + \lgth{y}.
  \qedhere
  \end{align*}
\end{proof}

We recall the notation $H_r(x,y)=G_r(x,y)G_r(y,x)$
from~\eqref{eq_defHr}.

\begin{lem}
\label{lem_GRbounded}
There exists $C>0$ such that, for any $k\in \N$, $\sum_{x\in \Sbb_k}
H_R(e,x) \leq C$.
\end{lem}
\begin{proof}
Fix $r<R$. Write $u_k(r) = \sum_{x\in \Sbb_k} H_r(e,x)$. To $x\in
\Sbb_k$ and $y\in \Sbb_\ell$ one can associate thanks to the previous
lemma a point $\Phi(x,y)= xay \in \bigcup_{k+\ell\leq i \leq
k+\ell+C} \Sbb_i$. By~\eqref{eq_trivial_ancona}, we have
  \begin{align*}
  H_r(e,x)H_r(e,y) & \leq C H_r(e,x)H_r(e,a)H_r(e,y)
  = C H_r(e,x) H_r(x, xa) H_r(xa, xay)
  \\&
  \leq C H_r(e, xay).
  \end{align*}

Let us estimate the number of preimages under $\Phi$ of some point
$z$. Let $\gamma_z$ be a geodesic segment from $e$ to $z$. If $z=xay$
and $x$ is far away from $\gamma_z$, a tree approximation shows that
$\lgth{z}$ is significantly smaller than
$\lgth{x}+\lgth{a}+\lgth{y}$. This is impossible by construction.
Therefore, $x$ is contained in a ball of fixed radius $B(\gamma_z(k),
C)$. In particular, the number of possibilities for $x$ is uniformly
bounded. Arguing in the same way for $y$, we deduce that, for some
$C>0$, each point has at most $C$ preimages under $\Phi$.

Finally,
  \begin{align*}
  u_k(r) u_\ell(r)
  &= \sum_{x\in \Sbb_k, y\in \Sbb_\ell} H_r(e,x) H_r(e,y)
  \leq C \sum_{x\in \Sbb_k, y\in \Sbb_\ell} H_r(e, \Phi(x,y))
  \\&
  \leq C \sum_{i=k+\ell}^{k+\ell+C} \sum_{z\in \Sbb_i} H_r(e,z)
  \leq C \sum_{i=k+\ell}^{k+\ell+C} u_i(r).
  \end{align*}
As $r<R$, the sum $\sum_{x\in \Gamma} H_r(e,x)$ is finite
by~\eqref{eq_eta}. In particular, the sequence $u_k(r)$ is summable,
and reaches its maximum $M(r)$ at some index $k_0(r)$. Using the
previous equation with $k=\ell=k_0(r)$, we get $M(r)^2 \leq
C(C+1)M(r)$, hence $M(r)\leq C(C+1)=D$.

Finally, for every $r<R$, for every $k\in \N$, one has $\sum_{x\in
\Sbb_k} H_r(e,x) \leq D$. The lemma follows by letting $r$ tend to
$R$.
\end{proof}

The following lemma is the main estimate in the proof of
Theorem~\ref{thm_ancona}. It gives superexponentially small estimates
for the $R$-probabilities of paths staying too far away from
geodesics, implying that such paths are very unlikely and will not
contribute a lot to $G_R$.
\begin{lem}
\label{lem_avoidballs}
Assume that $\mu$ is finitely supported and symmetric. There exist
$n_0>0$ and $\epsilon>0$ such that, for any $n\geq n_0$, for any
$x,y,z\in \Gamma$ on a geodesic segment (in this order) with
$d(x,y)\geq n$ and $d(y,z)\geq n$,
  \begin{equation*}
  G_R(x,z; B(y,n)^c) \leq 2^{-e^{\epsilon n}}.
  \end{equation*}
\end{lem}
\begin{proof}
Without loss of generality, one can assume $y=e$.

Fix some $\epsilon>0$ very small, and let $N=\lfloor e^{\epsilon n}
\rfloor$. In this proof, we will write $C$ for a generic constant
independent of $\epsilon$. The idea of the proof is to construct $N$
\emph{barriers} $A_1,\dotsc, A_N$ such that any trajectory of the
random walk going from $x$ to $z$ outside of $B(e,n)$ has to go
through $A_1$, then $A_2$, and so on. Decomposing a trajectory
according to its first visit to $A_1$, then $A_2$, and so on, we
obtain as in~\eqref{eq_Grcompose}
  \begin{equation}
  \label{eq_mlwkxcvoiu}
  G_R(x,z; B(e,n)^c) \leq \sum_{a_1\in A_1} \dotsm \sum_{a_N\in A_N}
  G_R(x, a_1) G_R(a_1,a_2)\dotsm G_R(a_{N-1}, a_N) G_R(a_N, z).
  \end{equation}
We will construct the barriers so that, writing $A_0=\{x\}$ and
$A_{N+1}=\{z\}$, one has for any $0\leq i \leq N$
  \begin{equation}
  \label{eq_proba_between_barriers}
  \sum_{a\in A_i} \sum_{b\in A_{i+1}} G_R(a, b)^2 \leq 1/4.
  \end{equation}
This implies the desired estimate on $G_R(x, z; B(y,n)^c)$ by
Cauchy-Schwarz, as follows. To write it formally, it is more
convenient to express things in terms of operators, as in
\cite{ledrappier_hyperbolic}. Define an operator $L_i:
\ell^2(A_{i+1}) \to \ell^2(A_i)$ by $L_i f(a)=\sum_{b\in A_{i+1}}
G_R(a,b) f(b)$. The sum to estimate in~\eqref{eq_mlwkxcvoiu} is
$(L_0\dotsm L_N \delta_z)(x)$, it is therefore bounded by $\prod
\norm{L_i}$. Moreover,
  \begin{align*}
  \norm{L_i f}_{\ell^2}^2
  & = \sum_{a\in A_i} \abs{\sum_{b\in A_{i+1}} G_R(a,b) f(b)}^2
  \leq\sum_{a\in A_i} \left(\sum_{b\in A_{i+1}} G_R(a,b)^2\right) \cdot \left( \sum_{b\in A_{i+1}} \abs{f(b)}^2\right)
  \\&
  = \left(\sum_{a\in A_i}\sum_{b\in A_{i+1}} G_R(a,b)^2 \right) \norm{f}_{\ell^2}^2.
  \end{align*}
With~\eqref{eq_proba_between_barriers}, we obtain $\norm{L_i} \leq
\left(\sum_{a\in A_i}\sum_{b\in A_{i+1}} G_R(a,b)^2 \right)^{1/2}
\leq 1/2$, and the result of the lemma follows.

It remains to construct barriers
satisfying~\eqref{eq_proba_between_barriers}. The construction is
geometric, and is done in the hyperbolic space $\HH^m$ for some
$m\geq 2$ (or rather its model as the euclidean unit ball in
$\R^{m}$, with the boundary at infinity identified with the unit
sphere $S^{m-1}$ in $\R^{m}$). By \cite{bonk_schramm}, the group
$\Gamma$ with its word metric is roughly similar to a subset of such
a space: if $m$ is large enough, there exists a mapping $\Psi: \Gamma
\to \HH^m$ and $\lambda>0$, $C>0$ such that $\abs{\lambda
d_{\HH}(\Psi(u), \Psi(v)) - d(u,v)}\leq C$ for all $u,v\in \Gamma$.
The image under $\Psi$ of a geodesic in $\Gamma$ is a quasi-geodesic
in $\HH^m$, therefore it remains uniformly close to a true hyperbolic
geodesic (see for instance~\cite[Theorem 5.11]{ghys_hyperbolique}).
It follows that it does not make a serious difference to use
geodesics in $\Gamma$ or in $\HH^m$.

The hyperbolic geodesic from $\Psi(x)$ to $\Psi(z)$ can be extended
biinfinitely. Composing with an hyperbolic isometry, we can assume
that this geodesic goes through the center $O$ of the ball model of
$\HH^m$, and that $\Psi(e)$ is a bounded distance away from $O$. Let
$\xi$ be the endpoint of this hyperbolic geodesic in negative time.
To an angle $\theta \in [0,\pi]$, we associate the union of all the
semiinfinite geodesics $[O\zeta)$ (with $\zeta\in S^{m-1}$) making an
angle $\theta$ with $[O\xi)$ (its boundary at infinity is the circle
of points at distance $\theta$ of $\xi$ in $S^{m-1}$). Let then
$A(\theta)$ be the set of points $a$ in $B(e,n)^c\subset \Gamma$ such
that $\Psi(a)$ is at a distance at most $C_0$ of such a geodesic. If
$C_0$ is chosen large enough, a path of the random walk going from
$x$ to $z$ in $B(e,n)^c$ can not jump over $A(\theta)$ since $\mu$
has finite support, so that $A(\theta)$ is a barrier.

In $X=[0,\pi]$, consider $X_i = [(2i-1)/N, 2i/N]$ (for $1\leq i\leq
N$). Those intervals are separated by $1/N\sim e^{-\epsilon n}$. In
each of them, we will choose an angle $\theta_i$ and let
$A_i=A(\theta_i)$. One should then ensure
that~\eqref{eq_proba_between_barriers} is satisfied. To do so, we
will choose each $\theta_i$ at random as follows. Let $\Omega =
\prod_{i=1}^N X_i$, endowed with the product of the probability
measures $N\dLeb$ on $X_i$. Define a function $f_i$ on $\Omega$ by
  \begin{equation*}
  f_i(\theta_1,\dotsc,\theta_N) = \sum_{a\in A(\theta_i), b\in A(\theta_{i+1})} G_R(a,b)^2,
  \end{equation*}
where by convention $A(\theta_0)=\{x\}$ and $A(\theta_{N+1}) =
\{z\}$. One should find a value of
$\bartheta=(\theta_1,\dotsc,\theta_N)$ such that $f_i(\bartheta) \leq
1/4$ for all $i$. We will show that
  \begin{equation}
  \label{eq_intfi}
  \int f_i \leq C e^{-\rho n},
  \end{equation}
for some $\rho>0$ independent of $\epsilon$. It follows that $\int
(\sum f_i) \leq C(1+N) e^{-\rho n} \leq C(1+e^{\epsilon n})e^{-\rho
n}$. Choosing $\epsilon$ small enough, this is exponentially small,
and is in particular bounded by $1/4$ for large enough $n$. This
yields a point $\bartheta$ with $\sum f_i(\bartheta) \leq 1/4$, for
which the corresponding barriers
satisfy~\eqref{eq_proba_between_barriers}.

Let us now prove~\eqref{eq_intfi}. We will only give the argument for
$1\leq i\leq N-1$: the case of $f_0$ and $f_N$ is slightly different
(since $A_0=\{x\}$ and $A_{N+1}=\{z\}$ are fixed), it turns out to be
analogous to the general case, but simpler. Fix some $i\in [1,N-1]$.
To each $a\in \Gamma$ and $j\in \{i, i+1\}$, we associate the set
$X_j(a)$ of angles $\theta$ in $X_j$ such that $a\in A_j(\theta)$. By
definition,
  \begin{equation*}
  \int f_i = \sum_{a,b \in \Gamma} G_R(a,b)^2\cdot  N \Leb(X_i(a)) \cdot N\Leb(X_{i+1}(b)).
  \end{equation*}
For $a\in \Sbb_k$, its image under $\Psi$ is at distance at least
$\alpha k$ of $O$ in $\HH^m$, for some $\alpha>0$. If one moves away
from this point by at most $C_0$, the visual angle from $O$ varies by
at most $C e^{-\alpha k}$. It follows that $\Leb(X_j(a)) \leq C
e^{-\alpha k}$. Since $N\leq e^{\epsilon n}$, we obtain
  \begin{equation*}
  \int f_i \leq C e^{2\epsilon n} \sum_{a,b} G_R(a,b)^2 e^{-\alpha \lgth{a}} e^{-\alpha \lgth{b}},
  \end{equation*}
where the sum is restricted to those $a$ and $b$ outside of $B(e,n)$
and whose images under $\Psi$ belong to the $C_0$-neighborhoods of
the sectors delimited respectively by $X_i$ and $X_{i+1}$. Writing
$u=a^{-1}b$ (with $\lgth{u}\leq \lgth{a}+\lgth{b}$), we get
  \begin{equation*}
  \int f_i \leq C e^{2\epsilon n} \sum_u G_R(e,u)^2 e^{-\alpha \lgth{u}} N(u),
  \end{equation*}
where $N(u)$ is the number of ways to decompose $u$ as $a^{-1}b$. Fix
a point $u$, and such a decomposition $u=a^{-1}b$.

The hyperbolic geodesics from $O$ to, respectively, $\Psi(a)$ and
$\Psi(b)$, make an angle at least $e^{-\epsilon n}/2$. Therefore,
they are far away from each other outside of the ball $B(O, 2\epsilon
n)$. It follows from a tree approximation that
  \begin{equation*}
  d_{\HH}(\Psi(a), \Psi(b)) \geq \lgth{\Psi(a)} + \lgth{\Psi(b)} - 4\epsilon n -C.
  \end{equation*}
Since $\Psi$ is a quasisimilarity, we deduce that $\lgth{u} = d(a,b)
\geq \lgth{a}+\lgth{b} - C\epsilon n - C$. In particular, if
$\epsilon$ is small enough, since $\lgth{a}\geq n$ and $\lgth{b}\geq
n$, we obtain $\lgth{u}\geq n$. It also follows from this argument
that a geodesic in the group from $a$ to $b$ has to pass through the
ball $B(e, C\epsilon n)$, since geodesics in the group and in
hyperbolic space remain a bounded distance away. Let $\gamma$ be a
geodesic segment from $e$ to $u$ in $\Gamma$, then $a \gamma$ is a
geodesic segment from $a$ to $b$. There exists a time $j$ such that
$a \gamma(j) \in B(e, C\epsilon n)$. Finally, $a\in
\bigcup_{j=0}^{\lgth{u}} \gamma(j)^{-1} B(e, C\epsilon n)$, which
gives at most $(\lgth{u}+1) C^{C\epsilon n}$ possibilities for $a$.
Arguing similarly for $b$, be obtain $N(u) \leq (\lgth{u}+1)^2
e^{C\epsilon n}$ for some $C>0$.

Finally, we have
  \begin{equation*}
  \int f_i \leq C e^{2\epsilon n} \sum_{\lgth{u}\geq n} G_R(e,u)^2 e^{-\alpha\lgth{u}} (\lgth{u}+1)^2 e^{C\epsilon n}
  \leq C e^{(C+2)\epsilon n} \sum_{\lgth{u}\geq n} G_R(e,u)^2 e^{-\alpha \lgth{u}/2}.
  \end{equation*}
Since $\sum_{\lgth{u}=k} G_R(e,u)^2$ equals $\sum_{\lgth{u}=k}
H_R(e,u)$ by symmetry of $\mu$, it is uniformly bounded by
Lemma~\ref{lem_GRbounded}. Therefore, $\int f_i$ is bounded by $C
e^{C\epsilon n} e^{-\alpha n/2}$. If $\epsilon$ is small enough, this
is at most $Ce^{-\alpha n/4}$. This proves~\eqref{eq_intfi} and
concludes the proof of the lemma.
\end{proof}

The following lemma is proved in~\cite{gouezel_lalley}, and is
elementary (see the proof of Theorem~4.1 there):
\begin{lem}
\label{lem_abstract_ancona}
Let $\mu$ be an admissible measure on a Gromov-hyperbolic group.
Assume that, for all $K>0$, there exists $n_0$ such that, for all
$n\geq n_0$, for all points $x,y,z$ on a geodesic segment (in this
order) with $d(x,y)\in [n,100n]$ and $d(y,z)\in [n,100 n]$, one has
$G_R(x,z; B(y,n)^c) \leq K^{-n}$. Then $\mu$ satisfies uniform Ancona
inequalities. It even satisfies strong uniform Ancona inequalities
(as defined below in Definition~\ref{def_strong_ancona}).
\end{lem}
To prove this lemma, one uses recursively its assumptions to show
that most $r$-weight is concentrated on paths staying close enough to
the geodesic from $x$ to $z$, and in particular passing in a ball of
fixed radius around $y$. This lemma, together with
Lemma~\ref{lem_avoidballs}, proves uniform Ancona inequalities for
symmetric measures, i.e., Theorem~\ref{thm_ancona}. Strong uniform
Ancona inequalities (see below) are then deduced as in
done~\cite[Theorem 4.6]{gouezel_lalley}.

We need the following strengthening of Ancona inequalities:

\begin{definition}
\label{def_strong_ancona}
A measure $\mu$ on a Gromov-hyperbolic group satisfies strong uniform
Ancona inequalities if it satisfies uniform Ancona inequalities and,
additionally, there exist constants $C>0$ and $\rho>0$ such that, for
all points $x,x',y,y'$ whose configuration is approximated by a tree
as follows
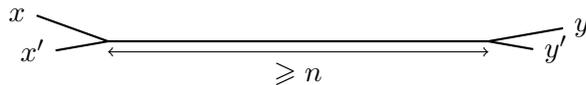
\begin{center}
  \begin{tikzpicture}
  \draw[thick]
    (0,0) -- +(160:1cm)   node[left] {$x$}
    (0,0) -- +(190:0.7cm) node[left] {$x'$}
    (0,0) -- (5,0) coordinate (y0)
    (y0) -- +(10:1cm)     node[right] {$y$}
    (y0) -- +(-10:0.6cm) node[right] {$y'$};
  \draw[<->] (0,-0.14) -- node[below] {$\geq n$} (5, -0.14);
  \end{tikzpicture}
\end{center}
for any $r\in [1,R]$,
  \begin{equation}
  \label{eq_strong_ancona}
  \abs{ \frac{G_r(x,y)/G_r(x',y)}{G_r(x,y')/G_r(x',y')} -1} \leq C e^{-\rho n}.
  \end{equation}
\end{definition}

Ancona inequalities ensure that the quantity
$\frac{G_r(x,y)/G_r(x',y)}{G_r(x,y')/G_r(x',y')}$ in the definition
is bounded from above and from below. Strong Ancona inequalities
ensure that this quantity is exponentially close to $1$ in terms of
the distance between the sets of points $\{x,x'\}$ and $\{y,y'\}$.
These bounds are not formal consequences of Ancona inequalities, but
they are consequences of Ancona inequalities in suitable domains
(that follow from Lemma~\ref{lem_avoidballs}). Applying
Lemmas~\ref{lem_avoidballs} and~\ref{lem_abstract_ancona}, we obtain
the following result, strengthening Theorem~\ref{thm_ancona}.

\begin{thm}
\label{thm_ancona_strong}
If $\mu$ is admissible, finitely supported and symmetric on a
non-elementary Gromov-hyperbolic group, it satisfies strong uniform
Ancona inequalities.
\end{thm}

When one takes $x'=e$, then the quantity appearing
in~\eqref{eq_strong_ancona} is the ratio $K_{r,y}(x)/K_{r,y'}(x,)$ of
the Martin kernels. The theorem implies that, when $y_i$ tends to a
point $\xi\in \partial \Gamma$, the sequence $K_{r,y_i}(x)$ is a
Cauchy sequence (since the points $x,e$ and $y_i,y_j$ satisfy the
assumptions of the definition with a large $n$ for large enough
$i,j$). Hence, it converges to a function $K_{r,\xi}(x)$. This is the
main step in the proof that the Martin boundary for $r$-harmonic
functions coincides with the geometric boundary (one should also
check that $K_{r,\xi} \not=K_{r,\eta}$ for $\xi\not=\eta$, which is
easy). We omit the (classical) details, see for instance
\cite{izumi_hyperbolic}.

\section{Asymptotics of the Green function}
\label{sec_diff_eq}

Let $\Gamma$ be a non-elementary Gromov-hyperbolic group. Our goal in
this section is to prove the following theorem.

\begin{thm}
\label{thm_strong_asymp}
Let $\mu$ be an admissible probability measure on $\Gamma$ satisfying
strong uniform Ancona inequalities. For any $x,y\in \Gamma$, there
exists $C(x,y)>0$ such that, when $r$ tends to $R=R(\mu)$,
  \begin{equation*}
  \partial G_r(x,y)/\partial r \sim C(x,y) (R-r)^{-1/2}.
  \end{equation*}
\end{thm}

Throughout this section, we fix a measure $\mu$ satisfying the
assumptions of this theorem. Theorem~\ref{thm_ancona_strong} shows
that it is the case for finitely supported symmetric measures, but
symmetry will play no additional role in this section. We will
concentrate mainly on the proof of Theorem~\ref{thm_strong_asymp} for
$x=y=e$, since the general case will follow easily.

In a sense, the proof of Theorem~\ref{thm_strong_asymp} is
essentially done in \cite{gouezel_lalley}, but there is an important
technical difference: a (well chosen) Markov automaton for a surface
group is transitive, while there can be several components in a
general hyperbolic group. This means that, in the thermodynamic
formalism, we will have to deal with several dominating components.
This problem is solved thanks to a technique of Calegari and Fujiwara
\cite{calegari_fujiwara} and to Lemma~\ref{lem_GRbounded}. This
sketch of the argument might be sufficient for experts, but since
there are several technical subtleties we will give most details
below. There is a significant overlap with some arguments in
\cite{gouezel_lalley}, but this seems necessary to keep the argument
understandable. Two significant differences
with~\cite{gouezel_lalley} (in addition to the existence of several
dominating components in the automaton, and directly related to this
issue) are that we need some a priori estimates (proved in
Subsection~\ref{subsec_apriori}), and that for $r<R$ we will
associate to $G_r$ a measure living on the group, not on the
boundary.

\subsection{A priori estimates}
\label{subsec_apriori}

The main idea behind the proof of Theorem~\ref{thm_strong_asymp}, as
in \cite{gouezel_lalley}, is that the function $G_r(e,e)$ almost
satisfies a differential equation. By~\eqref{eq_derive}, its
derivative with respect to $r$ is essentially $\sum_x
G_r(e,x)G_r(x,e) = \sum_x H_r(e,x)$ (recall the
notation~\eqref{eq_defHr}), and its second derivative is essentially
$\sum_{x,y} G_r(e,y) G_r(y,x) G_r(x,e)$. To prove that $G_r(e,e)$
almost satisfies a differential equation, we should relate those
quantities. The next proposition gives such a (crude) relation.

\begin{prop}
\label{prop_aprioridiff}
There exists $C>0$ such that, for all $r\in [1,R)$,
  \begin{equation*}
  C^{-1} \leq \frac{ \sum_{x,y} G_r(e,y) G_r(y,x)G_r(x,e)}{ \left(\sum_x G_r(e,x)G_r(x,e)\right)^3} \leq C.
  \end{equation*}
\end{prop}
\begin{proof}
Consider two points $x,y$. The triangle with vertices $e,x,y$ is
thin, so there exists a point $w$ (defined uniquely up to a finite
set) which is close to each of its sides. By the Ancona inequality,
we have $G_r(e,y)=G_r(w^{-1}, w^{-1}y) \leq C G_r(w^{-1}, e) G_r(e,
w^{-1}y)$, since a geodesic segment from $w^{-1}$ to $w^{-1}y$ passes
close to $e$ by construction. Similar estimates hold along $[y,x]$
and $[x,e]$, and we obtain
  \begin{multline*}
  G_r(e,y) G_r(y,x)G_r(x,e)\\
  \leq C G_r(w^{-1}, e)G_r(e,w^{-1}) \cdot G_r(w^{-1}x, e)G_r(e, w^{-1}x) \cdot G_r(w^{-1}y, e)G_r(e, w^{-1}y).
  \end{multline*}
The points $w^{-1}, w^{-1}x$ and $w^{-1}y$ determine $x$ and $y$.
Using the notation $H_r$ and summing over $x$ and $y$, we get
  \begin{equation*}
  \sum_{x,y} G_r(e,y) G_r(y,x)G_r(x,e) \leq C \sum_{a,b,c} H_r(e,a) H_r(e,b) H_r(e,c).
  \end{equation*}
This is one of the inequalities of the proposition.

For the reverse inequality, for any $u\in \Gamma$, write $B(u)$ for
the set of points $x$ such that a geodesic from $e$ to $x$ passes
close to $u$. Lemma~\ref{lem_rallonge} ensures that, for any $z\in
\Gamma$, there exists a uniformly bounded $a$ such that $uaz\in
B(u)$. In particular, Harnack inequalities~\eqref{eq_harnack} give
$G_r(e,z) \leq C(u) G_r(e,uaz)$ (and $H_r$ satisfies the same
inequality). Hence, $\sum_{z\in \Gamma} H_r(e,z) \leq C(u) \sum_{v\in
B(u)} H_r(e,v)$. Choose now three geodesic segments $\gamma_1$,
$\gamma_2$ and $\gamma_3$ (with endpoints denoted by $u_1, u_2,
u_3$), long enough and going in three different directions (this is
possible since the group is non-elementary) so that the sets $B(u_i)$
are pairwise disjoint, and so that a geodesic from $B(u_i)$ to
$B(u_j)$ ($i\not=j$) has to pass close to $e$. We get
  \begin{equation*}
  \left(\sum_z H_r(e,z)\right)^3 \leq C(u_1, u_2, u_3)
  \sum_{v_i \in B(u_i)} H_r(e,v_1) H_r(e,v_2) H_r(e,v_3).
  \end{equation*}
By~\eqref{eq_trivial_ancona}, we have $G_r(v_1,e) G_r(e,v_2) \leq C
G_r(v_1, v_2)$, and similarly for circular permutations. This sum is
therefore bounded by $\sum G_r(v_1, v_2) G_r(v_2, v_3) G_r(v_3,
v_1)$. Let $y=v_1^{-1}v_2$ and $x=v_1^{-1} v_3$. The point close to
the three sides of a geodesic triangle with vertices $e$, $x$ and $y$
is close to $v_1^{-1}$ by construction, hence $x$ and $y$ determine
$v_1$ (and then $v_2$ and $v_3$) up to a finite number of
possibilities. We finally get
  \begin{equation*}
  \left(\sum_z H_r(e,z)\right)^3 \leq C \sum_{x,y} G_r(e,y) G_r(y,x)G_r(x,e),
  \end{equation*}
proving the other inequality of the lemma.
\end{proof}

\begin{cor}
\label{cor_crude_eta}
There exist $A\geq 0$ and $C>0$ such that, for all $r\in [1,R)$,
  \begin{equation*}
  \frac{C^{-1}}{\sqrt{A+(R-r)}} \leq \sum_{x\in \Gamma} H_r(e,x) \leq \frac{C}{\sqrt{A+(R-r)}}.
  \end{equation*}
Moreover, $A=0$ if and only if $\sum_{x\in \Gamma} H_R(e,x) =
+\infty$.
\end{cor}
\begin{proof}
Let $\eta(r)=\sum_x H_r(e,x)$, and $F(r) = r^2 \eta(r)$.
By~\eqref{eq_derive},
  \begin{equation*}
  F'(r) = 2r\sum_{x,y} G_r(e,y)G_r(y,x)G_r(x,e).
  \end{equation*}
Therefore, Proposition~\eqref{prop_aprioridiff} shows that
$2F'(r)/F(r)^3 = (-1/F(r)^2)'$ is bounded from above and below.
Integrating this estimate on an interval $[r,s]$, we get
  \begin{equation*}
  C^{-1}(s-r) \leq 1/F(r)^2 - 1/F(s)^2 \leq C(s-r).
  \end{equation*}
When $s$ increases to $R$, $F(s)$ converges either to a positive
constant or to infinity. Hence, $1/F(s)^2$ converges to $A\in
[0,\infty)$. We get $A+C^{-1}(R-r) \leq 1/F(r)^2 \leq A+C(R-r)$, from
which the corollary follows.
\end{proof}
We shall see later that $A$ is in fact equal to $0$. Therefore, this
corollary gives the right order of magnitude $1/\sqrt{R-r}$ for the
function $\eta(r) = \sum_{x\in \Gamma} H_r(e,x)$. However, to obtain
Theorem~\ref{thm_strong_asymp}, we need to get \emph{asymptotics}, of
the form $\eta(r)\sim C/\sqrt{R-r}$. The strategy will be the same,
relying on the differential equation, but we will need to improve
Proposition~\ref{prop_aprioridiff}, to get convergence instead of
mere bounds. This is most conveniently done using the transfer
operator on a Markov automaton, as we will explain in the next
subsection. Before doing this, let us state a final technical lemma,
that relies on Corollary~\ref{cor_crude_eta} and will be important
later on.

\begin{lem}
\label{lem_C0bound}
Fix $a\in \Gamma$. There exists $C>0$ such that, for any $x\in
\Gamma$, for any $r\in [1,R]$,
  \begin{equation*}
  \abs{ \log\left(\frac{G_r(e,x)}{G_r(a,x)}\right) - \log\left(\frac{G_R(e,x)}{G_R(a,x)}\right)}
  \leq C \sqrt{R-r}
  \end{equation*}
and
  \begin{equation*}
  \abs{ \log\left(\frac{G_r(x,e)}{G_r(x,a)}\right) - \log\left(\frac{G_R(x,e)}{G_R(x,a)}\right)}
  \leq C \sqrt{R-r}.
  \end{equation*}
\end{lem}
\begin{proof}
The second estimate of the lemma can be deduced from the first one
applied to the measure $\check{\mu}(g)=\mu(g^{-1})$, we will
therefore concentrate on the first one.

Fix some $x\in \Gamma$. Let $f(r) = \log(G_r(e,x)/G_r(a,x))$, we will
show that its derivative is bounded in absolute value by
$C/\sqrt{A+R-r}$, where $C$ is a constant that does not depend on $x$
(of course, it may depend on $a$). By integration, this gives
$\abs{f(r)-f(R)} \leq C(\sqrt{A+R-r} - \sqrt{A})$, which is bounded
by $C \sqrt{R-r}$ as desired.

We write $f(r) = \log (rG_r(e,x)) - \log(rG_r(a,x))$. With the
formula~\eqref{eq_derive} for the derivative of $rG_r$, we get
  \begin{align*}
  f'(r)
  &= \frac{\sum_y G_r(e,y) G_r(y,x)}{r G_r(e,x)} - \frac{\sum_y G_r(a,y) G_r(y,x)}{r G_r(a,x)}
  \\&
  = \frac{1}{r} \sum_y \left(1-\frac{G_r(a,y)/G_r(e,y)}{G_r(a,x)/G_r(e,x)}\right)
  \frac{ G_r(e,y)G_r(y,x)}{G_r(e,x)}.
  \end{align*}
Consider a geodesic segment $\gamma$ from $e$ to $x$ and write
$\gamma(n)$ ($0\leq n\leq \lgth{x}$) for the point on $\gamma$ at
distance $n$ of $e$. Let $\Gamma_n$ denote the set of points $y\in
\Gamma$ whose projection on $\gamma$ is $\gamma(n)$, i.e.,
$d(y,\gamma(n))\leq d(y,\gamma(i))$ for $i\neq n$. Note that there
can be several such projections -- in this argument, the multiplicity
is not important, otherwise one can avoid it by using only the first
projection. For $y\in \Gamma_n$, the points $e,a$ and $x,y$ are in
the configuration of Theorem~\ref{thm_ancona_strong}, with a
separating distance at least $n-C$ (for some $C$ only depending on
$a$). Applying this theorem, we obtain
  \begin{equation*}
  \abs{f'(r)}\leq C\sum_{n=0}^{\lgth{x}} \sum_{y\in \Gamma_n} e^{-\rho n} \frac{ G_r(e,y)G_r(y,x)}{G_r(e,x)}.
  \end{equation*}
For $y\in\Gamma_n$, geodesics from $e$ to $y$ and from $y$ to $x$
pass close to $\gamma(n)$. Hence, Ancona inequalities give
  \begin{equation*}
  G_r(e,y) G_r(y,x)\leq C G_r(e,\gamma(n)) G_r(\gamma(n), y) G_r(y,\gamma(n)) G_r(\gamma(n), x)
  \leq C G_r(e, x) H_r(\gamma(n), y).
  \end{equation*}
Finally,
  \begin{equation*}
  \abs{f'(r)}\leq C \sum_{n=0}^{\lgth{x}} \sum_{y\in \Gamma_n} e^{-\rho n} H_r(\gamma(n), y)
  \leq C \sum_{n=0}^{\lgth{x}} e^{-\rho n} \sum_{y\in \Gamma}  H_r(e, \gamma(n)^{-1}y).
  \end{equation*}
Since $\sum_{z\in \Gamma} H_r(e,z) \leq C/\sqrt{A+R-r}$ by
Corollary~\ref{cor_crude_eta} and $e^{-\rho n}$ is summable, this
proves the lemma.
\end{proof}

\subsection{Symbolic dynamics}
\label{subsec_symbolic}

For a nice introduction to the topics of this paragraph and the next
one, see \cite{calegari_fujiwara}.

Let $S$ be a finite symmetric generating set of the group $\Gamma$. A
rooted $S$-labeled automaton (or simply automaton) is a finite
directed graph $\mathcal{A}= (V,E,s_{*})$ with distinguished vertex
$s_{*}$ (``start''), and a labeling $\alpha: E \rightarrow S$ of
edges by generators of the group.

A \emph{path} in the graph is a sequence of edges
$e_0,\dotsc,e_{m-1}$ such that the endpoint of $e_i$ is the starting
point of $e_{i+1}$. To such a path $\gamma$, one can associate a path
$\alpha(\gamma)$ in the Cayley graph of $\Gamma$ by multiplying
successively the generators read along the edges of the path. Let
$\alpha_*(\gamma)$ be the endpoint of $\alpha(\gamma)$.

\begin{definition}\label{definition:automaticStructure}
An automaton is a strongly Markov automatic structure for $\Gamma$
if:
\begin{enumerate}
\item Every vertex $v\in V$ is accessible from the start state
    $s_{*}$.
\item For every path $\gamma$, the path $\alpha (\gamma)$ is a
    geodesic path in $\Gamma$.
\item The endpoint mapping $\alpha_{*}$ induced by $\alpha$ is a
    bijection of the set of paths starting at $s_*$ onto
    $\Gamma$.
\end{enumerate}
\end{definition}

In particular, the sphere $\Sbb_k$ of $\Gamma$ is in bijection with
the set of paths of length $k$ starting from $s_*$.

Every Gromov-hyperbolic group admits such a strongly Markov automatic
structure, by a theorem of Cannon~\cite{cannon}. Let us fix once and
for all such an automaton. An infinite path in the graph determines a
semi-infinite geodesic in the group starting from $e$, and therefore
a point on the boundary at infinity. In this way, we extend
$\alpha_*$ to a map from infinite paths to $\partial \Gamma$.

A component of the automaton is a maximal subset in which any vertex
can be reached from any other vertex. If there is a single
non-trivial component, the recurrent part is transitive. This is the
case for well-chosen automata for subgroups of $\SL$, but for general
hyperbolic groups there is no such transitivity. Identifying points
belonging to the same component, one obtains a new directed graph,
the components graph, in which there is no loop. This graph encodes
how different components interact.

We will denote by $\Sigma^*$ the set of finite paths in the graph, by
$\Sigma$ the set of semi-infinite paths, and $\bSigma = \Sigma^*\cup
\Sigma$. These sets are endowed with a metric $d(\omega,\omega')=
2^{-n}$ where $n$ is the first time the paths $\omega$ and $\omega'$
differ. With this metric, $\Sigma^*$ is a dense open subset of the
compact space $\bSigma$. The map $\alpha_*$ is continuous from
$\bSigma$ to $\Gamma\cup\partial \Gamma$.

Denote by $\boH^\hol$ the space of $\hol$-H\"older continuous functions
on $\bSigma$. For $0<\hol<\hol'$, one has the following basic
inequality (which is true in any metric space):
  \begin{equation}
  \label{eq_relate_holder}
  \norm{f}_{\boH^\hol} \leq 2\norm{f}_{C^0}^{1-\hol/\hol'}
  \norm{f}_{\boH^{\hol'}}^{\hol/\hol'}.
  \end{equation}
In particular, if a sequence of functions $f_n$ converges in $C^0$
and remains bounded in $\boH^{\hol'}$, then it converges in
$\boH^\hol$.

Note that an H\"older continuous function on $\Sigma^*$ uniquely
extends to an H\"older continuous function on $\bSigma$. Finally, let
$\sigma: \bSigma \to \bSigma$ be the left shift, forgetting the first
edge of a path.

\subsection{Peripheral spectrum of transfer operators}

Since the spectral description of transfer operators is very
classical, we will only sketch the proofs in this section, referring
to \cite{parry-pollicott} for more details.

Consider a finite directed graph $\boA$, let $\bSigma$ be the set of
finite or infinite paths in $\boA$, and let $\sigma$ be the left
shift. (If one is uncomfortable with the idea of considering finite
paths in the graph, one can equivalently add a cemetery to the graph,
that can be reached from any vertex, and extend a finite path by
infinitely many steps in the cemetery.) To any real-valued H\"older
continuous function $\phi:\bSigma \to \R$ (called a
\emph{potential}), one associates the so-called \emph{transfer
operator} $\boL_\phi$, defined on the set of H\"older continuous
functions by
  \begin{equation*}
  \boL_\phi f(\omega) = \sum_{\sigma(\omega')=\omega}e^{\phi(\omega')} f(\omega'),
  \end{equation*}
where for $\omega=\emptyset$ the empty path we only consider the
non-empty preimages of $\omega$. The iterates of this operator encode
a lot of information on the Birkhoff sums $S_n \phi(\omega) =
\sum_{j=0}^{n-1} \phi(\sigma^j \omega)$ of the potential $\phi$. For
instance, one has
  \begin{equation*}
  \boL_\phi^n 1(\emptyset) = \sum e^{S_n \phi(\omega)},
  \end{equation*}
where the sum is over all paths of length $n$.

In the case of hyperbolic groups, we will be interested in the
asymptotics of such sums, since for suitable potentials $\phi_r$ they
correspond to the sum of $H_r$ over the sphere of radius $n$ in
$\Gamma$ (this is one of the quantity we want to estimate precisely
to improve on Proposition~\ref{prop_aprioridiff}). Such asymptotics
can be read from the spectrum of $\boL_\phi$, that we now describe.

The simplest situation is when the graph is \emph{topologically
mixing}, i.e., one can go from any vertex to any other vertex (one
says that the graph is recurrent) and for any $a,b\in \boA$, for any
large enough $n$, there is a path of length exactly $n$ from $a$ to
$b$. In this case, the spectral description of $\boL_\phi$ is very
simple, and is given by the following theorem (called the
Ruelle-Perron-Frobenius theorem).
\begin{thm}
\label{thm_mixingRPF}
Assume $\boA$ is topologically mixing. The operator $\boL_\phi$
acting on the space of H\"older continuous functions has a unique
eigenvalue of maximal modulus denoted by $e^{\Press(\phi)}$, the rest
of its spectrum is contained in a disk of strictly smaller radius.
Moreover, the corresponding eigenfunction $h$ (suitably normalized)
is strictly positive everywhere, and the eigenprojector is given by
$\Pi f = \left(\int f\dd\lambda\right) h$ for some probability
measure $\lambda$ whose support is the set $\Sigma$ of infinite
paths. Finally, the probability measure $h\dd\lambda$ is invariant
under $\sigma$ and ergodic.
\end{thm}
In other words, one has
  \begin{equation*}
  \norm{\boL_\phi ^n f - e^{n\Press(\phi)}\left(\int f\dd\lambda\right) h}
  \leq C \norm{f} e^{-n\epsilon}e^{n\Press(\phi)},
  \end{equation*}
for some $C>0$ and $\epsilon>0$. This is Theorem~2.2 in
\cite{parry-pollicott} (the statement there is only given on
$\Sigma$, but the proofs readily adapt to $\bSigma$). The real number
$\Press(\phi)$ is called the \emph{pressure} of the potential $\phi$.

Assume now that $\boA$ is recurrent, but not mixing: there is a
minimal period $p>1$ such that any path from a vertex to itself has
length $np$ for some integer $n$. In this case, the set $V$ of
vertices of $\boA$ is a disjoint union $\bigsqcup_{j=0}^{p-1} V_j$,
where for any $j\in \Z/p\Z$ an outgoing edge of $V_j$ is an ingoing
edge of $V_{j+1}$ (we call this decomposition a cyclic decomposition
of $V$). Denoting by $\bSigma_j$ the set of paths beginning from a
vertex in $V_j$ and the empty path, then $\sigma$ maps $\bSigma_j$ to
$\bSigma_{j+1}$. Moreover, the restriction of $\sigma^p$ to any
$\bSigma_j$ is a topologically mixing subshift of finite type, to
which Theorem~\ref{thm_mixingRPF} applies. This readily implies that
the eigenvalues of maximal modulus of $\boL_\phi$ are of the form
$e^{2\ic k\pi/p}e^{\Press(\phi)}$ for some real number
$\Press(\phi)$, they are all simple, and the rest of the spectrum of
$\boL_\phi$ is contained in a disk of strictly smaller radius. More
specifically, there exist positive functions $h_j$ on $\bSigma_j$ and
probability measures $\lambda_j$ with support equal to $\Sigma_j$
such that
  \begin{equation*}
  \norm{\boL_\phi^n f -e^{n\Press(\phi)} \sum_{j=0}^{p-1}\left(\int f\dd\lambda_{(j-n \bmod p)}\right) h_j}
  \leq C \norm{f} e^{-n\epsilon}e^{n\Press(\phi)}.
  \end{equation*}

Assume finally that $\boA$ is not even recurrent. In this case, one
can associate to any component $\boC$ the restriction of $\phi$ to
paths staying in $\boC$ and the corresponding transfer operator
$\boL_{\boC}$. The previous description applies to $\boL_{\boC}$: it
has finitely many eigenvalues of maximal modulus
$e^{\Press_{\boC}(\phi)}$, they are of the form $e^{2\ic
k\pi/p_{\boC}} e^{\Press_{\boC}(\phi)}$ for some $k\in \Z/p_{\boC}
\Z$, and $\boL_{\boC}$ has a spectral gap. Let $\Press(\phi)$ be the
maximum of $\Press_{\boC}(\phi)$ over all components. We call a
component maximal if $\Press_{\boC}(\phi) = \Press(\phi)$. The
dominating terms in $\boL_\phi^n$ come from the maximal components.
We will say that $\phi$ is \emph{semisimple} if there is no directed
path from a maximal component to a different maximal component.
Otherwise, the eigenvalue $e^{\Press(\phi)}$ has non-trivial Jordan
blocks, which makes the precise spectral description more cumbersome.

\begin{lem}
\label{lem_jordan}
Consider some edge $e_0$, and let $k>0$ be such that there is a path
from $e_0$ to successively $k$ different maximal components. For any
nonnegative function $f$ with $f\geq 1$ on the set of paths starting
with $e_0$, one has $\boL_\phi^n f(\emptyset) \geq C n^{k-1}
e^{n\Press(\phi)}$.
\end{lem}

In the semisimple case, the asymptotics of $\boL_\phi^n$ can be
described as follows.
\begin{thm}
\label{thm_spectral_desc}
Assume that $\phi$ is semisimple. Denote by $\boC_1,\dotsc, \boC_I$
the maximal components, with corresponding period $p_i$, and consider
for each $i$ a cyclic decomposition $\boC_i=\bigsqcup_{j\in \Z/p_i
\Z} \boC_{i,j}$. There exist functions $h_{i,j}$ and measures
$\lambda_{i,j}$ with $\int h_{i,j} \dd\lambda_{i,j}=1$ such that
  \begin{equation*}
  \norm{\boL_\phi^n f -e^{n\Press(\phi)} \sum_{i=1}^I \sum_{j=0}^{p_i-1}\left(\int f\dd\lambda_{i, (j-n \bmod p_i)}\right) h_{i,j}}
  \leq C \norm{f} e^{-n\epsilon}e^{n\Press(\phi)}.
  \end{equation*}
The probability measures $\dd\mu_i = \frac{1}{p_i}\sum_{j=0}^{p_i-1}
h_{i,j} \dd\lambda_{i,j}$ are invariant under $\sigma$ and ergodic.

Denote by $\boC_{\to,i,j}$ the set of edges from which one can reach
$\boC_{i,j}$ with a path of length in $p_i\N$, and by
$\boC_{i,j,\to}$ the set of edges that can be reached from
$\boC_{i,j}$ by a path of length in $p_i\N$. The function $h_{i,j}$
is bounded from below on paths beginning by an edge in
$\boC_{i,j,\to}$ (and the empty path) and vanishes elsewhere. The
support of the measure $\lambda_{i,j}$ is the set of infinite paths
beginning in $\boC_{\to,i,j}$ with infinitely many coordinates in
$\boC_{i}$.
\end{thm}
\begin{proof}
The lemma and the theorem are basic linear algebra once
Theorem~\ref{thm_mixingRPF} is given. Indeed, one can decompose
$\boL$ into a sum of operators corresponding to edges in the
components graph. Since there is no loop in this graph, this is a
Jordan-blocks like decomposition, which readily gives the dominating
spectrum of $\boL$ from the dominating spectrum on each component.
The only nontrivial assertion is on the support of $h_{i,j}$ and
$\lambda_{i,j}$ in the theorem. When there is only one non-trivial
component and this component is mixing, the argument is given
in~\cite[Theorem 6.1]{gouezel_lalley}. It easily extends to the
general case.
\end{proof}

We will need the following simple lemma later on:
\begin{lem}
\label{lem_absolute_continuity}
Under the assumptions of the above theorem, let $\beta_i = \sum_j
\lambda_{i,j}$. Then $\sigma_* \beta_i$ is absolutely continuous with
respect to $\beta_i$.
\end{lem}
\begin{proof}
The measures $\lambda_{i,j}$ are constructed as eigenmeasures of the
operator $(\boL_\phi^*)^{p_i}$. More precisely, they satisfy
$\boL_\phi^* \lambda_{i,j} = e^{\Press(\phi)} \lambda_{i,(j-1 \bmod
p_i)}$. In particular, $\boL_{\phi}^* \beta_i = e^{\Press(\phi)}
\beta_i$.

Consider a cylinder $[\omega_0,\dotsc, \omega_n]$, i.e., the set of
paths that start with those symbols. The function $\boL_\phi
1_{[\omega_0,\dotsc, \omega_n]}$ is uniformly bounded on the image of
this cylinder under $\sigma$, i.e., $[\omega_1,\dotsc, \omega_n]$,
and it vanishes elsewhere. Hence,
  \begin{equation*}
  \beta_i([\omega_0,\dotsc, \omega_{n}])
  = e^{-\Press(\phi)} \boL_\phi^* \beta_i(1_{[\omega_0,\dotsc, \omega_{n}]}
  = e^{-\Press(\phi)} \beta_i(\boL_\phi 1_{[\omega_0,\dotsc, \omega_n]})
  \leq C \beta_i([\omega_1,\dotsc, \omega_n]).
  \end{equation*}
Since $\sigma^{-1}([\omega_1,\dotsc, \omega_n])$ is a finite union of
cylinders of the form $[\omega_0,\dotsc, \omega_{n}]$, we obtain
$\beta_i(\sigma^{-1}[\omega_1,\dotsc, \omega_n]) \leq C
\beta_i([\omega_1,\dotsc, \omega_n])$. As cylinders generate the
topology, it follows that $\beta_i(\sigma^{-1}A) \leq C\beta_i(A)$
for any measurable set $A$.
\end{proof}

Finally, we will need to describe what happens under perturbations of
the potential.

\begin{prop}
\label{prop_spectral_perturb}
Let $\phi\in \boH^\hol$ be a semisimple H\"older potential, with
maximal components $\boC_1,\dotsc,\boC_I$ and spectral description as
in Theorem~\ref{thm_spectral_desc}. There exist $\epsilon>0$ and
$C>0$ such that, for any $\psi$ which is small enough in $\boH^\hol$,
there exist functions $h_{i,j}^\psi$ and measures
$\lambda_{i,j}^\psi$ (with the same support as, respectively,
$h_{i,j}$ and $\lambda_{i,j}$) and numbers $\Press_i(\phi+\psi)$ with
  \begin{equation*}
  \norm{\boL_{\phi+\psi}^n f - \sum_{i=1}^I e^{n\Press_i(\phi+\psi)}
  \sum_{j=0}^{p_i-1}\left(\int f\dd\lambda^\psi_{i, (j-n \bmod p_i)}\right) h^\psi_{i,j}}
  \leq C \norm{f} e^{-n\epsilon}e^{n\Press(\phi)}.
  \end{equation*}
The maps $\psi\mapsto \Press_i(\phi+\psi)$, $\psi\mapsto
h_{i,j}^\psi$ and $\psi\mapsto \lambda_{i,j}^\psi$ are real analytic
from a small ball around $0$ in $\boH^\hol$ to, respectively, $\R$,
$\boH^\hol$ and the dual of $\boH^\hol$. Finally,
  \begin{equation*}
  \Press_i(\phi+\psi) = \Press(\phi) + \int \psi \dd\mu_i + O(\norm{\psi}^2_{\boH^\hol}),
  \end{equation*}
where $\dd\mu_i = \frac{1}{p_i}\sum_{j=0}^{p_i-1} h_{i,j}
\dd\lambda_{i,j}$.
\end{prop}
\begin{proof}
Let us first assume that the shift is topologically mixing. In this
case, the dominating eigenvalue $e^{\Press(\phi)}$ of $\boL_\phi$ is
simple. Simple isolated eigenvalues and the corresponding
eigenprojectors and eigenfunctions depend in an analytic way on the
operator, by classical perturbation theory \cite{kato_pe}. Moreover,
by semicontinuity of the spectrum, the perturbed operators
$\boL_{\phi+\psi}$ also have a spectral gap, uniformly in $\psi$
close enough to $0$. One gets
  \begin{equation*}
  \boL_{\phi+\psi}^n f = e^{n\Press(\phi+ \psi)}\left(\int f \dd\lambda^\psi\right) h^\psi
  + O(e^{-n\epsilon} e^{n\Press(\phi)})
  \end{equation*}
for some $\Press(\phi+\psi)$, $\lambda^\psi$ and $h^\psi$ that depend
analytically on $\psi$. This almost completes the proof of the
theorem in this case, it only remains to show that $\Press(\phi+\psi)
= \Press(\phi) + \int \psi h \dd\lambda + O(\norm{\psi}^2)$. By
analyticity, it is sufficient to show that the derivative of the
pressure at $0$ is given by the integral with respect to the measure
$h\dd\lambda$. This is \cite[Proposition 4.10]{parry-pollicott}.

The topologically transitive case readily reduces to the mixing case
by considering $\sigma^p$ where $p$ is the period.

In the general case, one obtains different pressures
$\Press_i(\phi+\psi)$ on each component $\boC_i$. On other components
that were not maximal for $\phi$, the pressure of $\phi+\psi$ remains
bounded away from $\Press(\phi)$. It follows that the maximal
components of $\phi+\psi$ are contained in those of $\phi$ if $\psi$
is small enough. In particular, $\phi+\psi$ is semisimple, and
$\Press(\phi+\psi) = \max \Press_i(\phi+\psi)$. Finally, the spectral
description of $\boL_{\phi+\psi}$ follows from the description on
each component $\boC_i$ separately.
\end{proof}

\subsection{Transfer operators in hyperbolic groups}

Let $\Gamma$ be a non-elementary Gromov-hyperbolic group, and $\mu$ a
probability measure satisfying strong Ancona inequalities. Consider a
strongly Markov automatic structure for $\Gamma$, given by a directed
graph $\mathcal{A}= (V,E,s_{*})$ and a labeling $\alpha: E
\rightarrow S$ of edges by generators of the group. We will use
freely the notations of Paragraph~\ref{subsec_symbolic}.

For $r\in [1,R]$, let us define a potential $\phi_r$ on the set
$\Sigma^*$ of finite paths in the automaton by
  \begin{equation*}
  \phi_r(\omega) = \log\left(\frac{H_r(e, \alpha_*(\omega))}{H_r(e, \alpha_*(\sigma\omega))}\right).
  \end{equation*}
Consider a path $\omega=\omega_0\dotsm \omega_{n-1}$ of length $n$,
then
  \begin{equation*}
  e^{S_n \phi_r(\omega)} = \frac{H_r(e, \alpha_*(\omega_0\dotsm \omega_{n-1}))}{H_r(e, e)}.
  \end{equation*}
Let $E_*$ be the set of edges starting from the vertex $s_*$ of the
graph $\mathcal{A}$, and let $1_{[E_*]}$ be the function equal to $1$
on paths starting with an edge in $E_*$, and $0$ elsewhere. Using the
language of transfer operators, we have
  \begin{align*}
  H_r(e,e) \cdot \boL^n_{\phi_r}1_{[E_*]}(\emptyset)
  &= H_r(e,e) \sum_{\omega=\omega_0\dotsm\omega_{n-1}} e^{S_n \phi_r(\omega)} 1(\omega_0\in E_*)
  \\&= \sum H_r(e, \alpha_*(\omega_0\dotsm \omega_{n-1}))1(\omega_0 \in E_*).
  \end{align*}
Since $\alpha_*$ induces a bijection between the paths of length $n$
starting from $s_*$ and the sphere $\Sbb_n$ of radius $n$ in
$\Gamma$, we obtain
  \begin{equation*}
  \sum_{x\in \Sbb_n} H_r(e,x) = H_r(e,e) \boL_{\phi_r}^n 1_{[E_*]}(\emptyset).
  \end{equation*}
Therefore, the spectrum of $\boL_{\phi_r}$ will give asymptotics for
$\sum_{x\in \Sbb_n} H_r(e,x)$. To be able to use the results of the
previous paragraph, one should check that $\phi_r$ is H\"older
continuous.

\begin{lem}
\label{lem_perturb_holder}
There exists $\hol>0$ such that, for any $r\in [1,R]$, the function
$\phi_r$ is H\"older continuous of exponent $\hol$ on $\Sigma^*$.
Therefore, it extends to an H\"older continuous function on $\bSigma$,
that we still denote by $\phi_r$. It satisfies
$\norm{\phi_r}_{\boH^\hol}\leq C$, uniformly in $r\in [1,R]$.
Moreover,
  \begin{equation}
  \label{eq_bound_holder_diff}
  \norm{\phi_r-\phi_R}_{\boH^\hol} \leq C (R-r)^{1/3}.
  \end{equation}
\end{lem}
\begin{proof}
Consider two finite paths $\omega$ and $\omega'$ with
$d(\omega,\omega')=2^{-n}<1$, so that they match up to length $n\geq
1$. In particular, $\omega_0=\omega'_0$. Let $x=\alpha_*(\omega)$,
$x'=\alpha_*(\omega')$ and $a=\alpha_*(\omega_0)$, so that
$\phi_r(\omega) = \log(H_r(e, x)/H_r(a,x))$ and $\phi_r(\omega') =
\log(H_r(e,x')/H_r(a,x'))$. The points $e,a$ and $x,x'$ are in the
situation of strong Ancona inequalities
(Definition~\ref{def_strong_ancona}) with a separating distance
$n-1$. Since $\mu$ satisfies strong uniform Ancona inequalities, it
follows that $\abs{\phi_r(\omega) -\phi_r(\omega')}\leq C e^{-\rho
n}$ for some $\rho>0$. Hence, for some $\hol'>0$, $\phi_r$ belongs to
$\boH^{\hol'}$ and is uniformly bounded in this space.

Lemma~\ref{lem_C0bound} implies that $\norm{\phi_r-\phi_R}_{C^0} \leq
C(R-r)^{1/2}$. Together with the uniform boundedness of $\phi_r$ in
$\boH^{\hol'}$, this shows that $\norm{\phi_r-\phi_R}_{\boH^\hol}\leq
C(R-r)^{1/3}$ if $\hol$ is small enough, by~\eqref{eq_relate_holder}.

Finally, we have proved all those inequalities on the space
$\Sigma^*$ of finite paths. Since H\"older continuous functions on
$\Sigma^*$ extend to H\"older continuous functions on $\bSigma$, the
result follows.
\end{proof}
\begin{rmk}
One could in fact show that $\norm{\phi_r-\phi_R}_{\boH^\hol} \leq C
(R-r)^{1/2}$ by mimicking the proof of Lemma~\ref{lem_C0bound} at the
level of H\"older exponents. Since this computation is lengthy
and~\eqref{eq_bound_holder_diff} will be sufficient for our purposes,
we omit it.
\end{rmk}

\begin{lem}
We have $\Press(\phi_R)=0$. Moreover, $\phi_R$ is semisimple.
\end{lem}
\begin{proof}
Suppose $\Press(\phi_R) < 0$. Then $\boL_{\phi_R}^n 1_{[E_*]}$ goes
to zero exponentially fast in the space of H\"older functions. In
particular, $\sum_{x\in \Sbb_n} H_R(e,x) = H_R(e,e)\boL_{\phi_R}^n
1_{[E_*]}(\emptyset)$ is exponentially small. One can use this
estimate to prove that the series $G_{R+\epsilon}(e,e)$ converges for
some $\epsilon>0$: this is the content of the proof of Proposition
7.1 in \cite{gouezel_lalley} (the proof is written for symmetric
measures, but it applies equally well in non-symmetric situations).
This is a contradiction since, by definition, $R$ is the radius of
convergence of the series $G_r(e,e)$. Hence, $\Press(\phi_R)\geq 0$.

If $\Press(\phi_R)$ were strictly positive, or $\Press(\phi_R)=0$ but
$\phi_R$ were not semisimple, then Lemma~\ref{lem_jordan} would imply
that $\boL_{\phi_R}^n 1_{[E_*]}(\emptyset)$ would tend to infinity.
This quantity is equal to $H_R(e,e)^{-1}\sum_{x\in \Sbb_n} H_R(e,x)$.
Since it remains bounded by Lemma~\ref{lem_GRbounded}, we obtain a
contradiction.
\end{proof}

One can now come back to Corollary~\ref{cor_crude_eta}. Since
$\Press(\phi_R)=0$ and $\phi_R$ is semisimple,
Theorem~\ref{thm_spectral_desc} implies in particular that
$\boL_{\phi_R}^n 1_{[E_*]}(\emptyset)$ is bounded from below. Since
it coincides with $H_R(e,e)^{-1}\sum_{x\in \Sbb_n} H_R(e,x)$, we get
$\sum_{x\in \Gamma} H_R(e,x) = +\infty$. This shows that the constant
$A$ in Corollary~\ref{cor_crude_eta} vanishes, and therefore
  \begin{equation}
  \label{eq_bound_Hr_crude}
  \frac{C^{-1}}{\sqrt{R-r}} \leq \sum_{x\in \Gamma} H_r(e,x) \leq \frac{C}{\sqrt{R-r}}.
  \end{equation}

Let us introduce a convenient notation: we will reserve the notation
$\tau(r)$ (possibly with some indices) for continuous functions of
$r$ taking values in $(0,+\infty)$ that extend continuously up to
$r=R$ and are bounded away from zero.

We will now use the spectral perturbation given by
Proposition~\ref{prop_spectral_perturb} to study $\boL_{\phi_r}$. If
$r$ is close to $R$, then $\phi_r-\phi_R$ is small in $\boH^\hol$ by
Lemma~\ref{lem_perturb_holder}. Applying the proposition on spectral
perturbation to the function $f=1_{[E_*]}$, we get the following. Let
$p$ be the least common multiple of the periods of the maximal
components of $\phi_R$. For any $q\in [0, p)$, one has (since
$\Press(\phi_R)=0$)
  \begin{equation*}
  \boL_{\phi_r}^{np + q} 1_{[E_*]}(\emptyset) = \sum_{i=1}^I e^{(np+q)\Press_i(\phi_r)} \tau_0(q,i,r) + O(e^{-n\epsilon}),
  \end{equation*}
for some functions $\tau_0(q,i,r)$ (as in the notation we introduced
in the previous paragraph). Since this is equal to
$H_r(e,e)^{-1}\sum_{x\in \Sbb_{np+q}} H_r(e,x)$ and since $\sum_{x\in
\Gamma} H_r(e,x) < \infty$, it follows in particular that
$\Press_i(\phi_r)$ is strictly negative for all $i$.

Summing over $n$ and $q$, we get
  \begin{align}
  \notag
  \sum_{x\in \Gamma}H_r(e,x)& = H_r(e,e) \sum_{n,q} \boL_{\phi_r}^{np + q} 1_{[E_*]}(\emptyset)
  = H_r(e,e)\sum_{q=0}^{p-1}\sum_{i=1}^I \frac{e^{q\Press_i(\phi_r)}}{1-e^{p\Press_i(\phi_r)}} \tau_0(q,i,r) + O(1)
  \\&
  \label{eq_sumHrasymp}
  = \sum_{i=1}^I \frac{\tau_1(i,r)}{\abs{\Press_i(\phi_r)}} + O(1),
  \end{align}
for some functions $\tau_1(i,r)$.

By~\eqref{eq_bound_Hr_crude}, $\abs{\Press(\phi_r)}=\inf_i
\abs{\Press_i(\phi_r)}$ is comparable to $\sqrt{R-r}$. It will be
important to show that all the $\abs{\Press_i(\phi_r)}$ are of the
same order of magnitude: otherwise, some components would not play a
significant role for $r<R$ while they would become important at
$r=R$, ruining the continuity properties we are seeking. This is the
main difference with the transitive situation, where there is only
one eigenvalue to consider.

\begin{thm}
\label{thm_Pressi_indep}
For any $i\in [1,I]$, the ratio $\Press_i(\phi_r)/\Press(\phi_r)$
tends to $1$ when $r\to R$.
\end{thm}
We will prove this theorem in the next subsection. It follows from
this result that
  \begin{equation}
  \label{eq_asymp_eta_Press}
  \sum_{x\in \Gamma} H_r(e,x) = \frac{\tau_2(r)}{\abs{\Press(\phi_r)}} + O(1).
  \end{equation}
Hence, the spectral data of $\boL_{\phi_r}$ are directly related to
the function $\eta(r)=\sum_{x\in \Gamma} H_r(e,x)$.

\subsection{Pressure does not depend on the component}

\label{subsec_integ_indep}
In this subsection, we prove Theorem~\ref{thm_Pressi_indep}. We will
in particular rely on the estimate $\norm{\phi_r-\phi_R}_{\boH^\hol}
\leq C(R-r)^{1/3}$ from~\eqref{eq_bound_holder_diff}, that in turn
was proved using the a priori estimates from Lemma~\ref{lem_C0bound}.

By Proposition~\ref{prop_spectral_perturb}, the variation of the
pressure mainly depends on the integral $\int (\phi_r - \phi_R)
\dd\mu_i$. We will show that this integral does not depend on $i$,
using a geometric argument in the group due to
\cite{calegari_fujiwara}.

Fix some $r\in [1,R]$. For $c\in \R$, we define a set $U(c) \subset
\partial \Gamma$ as the set of points $\xi$ such that, along some geodesic
from $e$ to $\xi$, $\log H_r(e,x) / d(e,x) \to c$. Equivalently, this
convergence holds along any geodesic tending to $\xi$, and one can
replace $H_r(e,x)$ with $H_r(a,x)$ and $d(e,x)$ with $d(b,x)$ for any
$a,b\in \Gamma$. Indeed, geodesics tending to $\xi$ remain within a
bounded distance from each other, by \cite[Proposition
7.2]{ghys_hyperbolique} (therefore, by Harnack inequalities $H_r$
varies by at most a multiplicative constant when one changes
geodesics), and the ratio $H_r(e,x)/H_r(a,x)$ also remains bounded
from above and from below again by Harnack inequalities. In
particular, $U(c)$ is invariant under the action of $\Gamma$: for any
$g\in \Gamma$, $g\cdot U(c) = U(c)$.

Let $c_i=\int  \phi_r \dd\mu_i$, we will show that, for all
$i\not=i'$, $g\cdot U(c_i)$ intersects $U(c_{i'})$ for some $g\in
\Gamma$. This will give $U(c_i)=U(c_{i'})$, hence $c_i=c_{i'}$ as
desired. To prove this, we will show that the sets $U(c_i)$ all have
positive measure for some measure on $\partial \Gamma$ which is
ergodic under the action of $\Gamma$.

Let us first construct the measure. Let $p$ be the least common
multiple of the periods $p_i$, and fix $q\in [0,p)$. It follows from
the spectral description of $\boL_{\phi_R}$
(Theorem~\ref{thm_spectral_desc}) that, for any H\"older continuous
function $f$ on $\bSigma$, $\boL_{\phi_R}^{np+q} f(\emptyset)$
converges when $n\to \infty$. In turn, this convergence follows for
any continuous function, by approximation (since the iterates of
$\boL_{\phi_R}$ on $C^0$ remain bounded, since $\boL_{\phi_R}^n 1$
itself remains bounded). If $f$ is a continuous function on
$\Gamma\cup \partial \Gamma$, then $\sum_{x\in \Sbb_n} H_R(e,x) f(x)
= H_R(e,e) \boL_{\phi_R}^n (1_{[E_*]}\cdot f\circ
\alpha_*)(\emptyset)$, and $f\circ\alpha_*$ is continuous. Let us
define a measure $m_n$ supported on $\Sbb_n$ by $m_n =\sum_{x\in
\Sbb_n} H_R(e,x)\delta_x$, this shows that the sequence of measures
$m_{np+q}$ converges to a limiting measure (which is supported on
$\partial \Gamma$, and has mass bounded from above and from below).
We deduce that the measures $(\sum_{1}^N m_n)/(\sum_1^N m_n(\Gamma))$
converge to a probability measure on $\partial \Gamma$, that we
denote by $\nu_R$. It also follows that this measure can be
constructed using the Patterson-Sullivan technique: the measures
  \begin{equation}
  \label{eq_defthetas}
  \theta_s = \sum_{x\in \Gamma} H_R(e,x) e^{-s \lgth{x}}\delta_x / \sum_{x\in \Gamma} H_R(e,x) e^{-s \lgth{x}}
  \end{equation}
are well defined for $s>0$, and they converge when $s$ tends to $0$
towards $\nu_R$.

For $g\in \Gamma$, let us denote by $L_g$ the left multiplication by
$\Gamma$. Then, for any $x\in\Gamma$,
  \begin{equation*}
  (L_g)_* \theta_s (x)= \theta_s(g^{-1}x) = \frac{H_R(e, g^{-1}x)}{H_R(e,x)} e^{-s(\lgth{g^{-1}x} - \lgth{x})} \theta_s(x)
  = \tilde K_x(g) e^{-s(\lgth{g^{-1}x} - \lgth{x})} \theta_s(x),
  \end{equation*}
where $\tilde K_x(g) = H_R(g,x)/H_R(e,x)$ is the Martin Kernel
associated to $H_R$. When $x$ tends to a point $\xi\in \partial
\Gamma$, this quantity converges to a limit denoted by $\tilde
K_\xi(g)$. Since $\lgth{g^{-1}x} - \lgth{x}$ is uniformly bounded
when $x$ varies in $\Gamma$, we deduce letting $s$ tend to $0$ that
  \begin{equation}
  \label{eq_quasiconformal}
  \frac{\dd(L_g)_* \nu_R}{\dd\nu_R}(\xi) = \tilde K_\xi(g).
  \end{equation}

Following the classical arguments of Patterson-Sullivan (due in this
context to~\cite{coornaert} and~\cite{BHM:2}), we deduce the
following:
\begin{prop}
\label{prop_ergodic}
The measure $\nu_R$ is ergodic for the action of $\Gamma$.
\end{prop}
\begin{proof}
We want to apply the results of~\cite{coornaert} and~\cite{BHM:2}
saying that a Patterson-Sullivan measure is ergodic. Thus, we should
interpret the function $\tilde K_\xi(g)$ in~\eqref{eq_quasiconformal}
as the exponential of a Busemann cocycle. Since $\tilde K_\xi(g)$ is
the limit of $H_R(g,x)/H_R(e,x)$, the function $\log \tilde K_\xi$
would be the Busemann cocycle associated to a distance $\tilde d$ if
$H_R(x,y) = C e^{-\tilde d(x,y)}$, for some constant $C$. Let us
therefore set $\tilde d(x,y) = -\log( F_R(x,y) F_R(y,x))$, where
$F_R$ is the first visit Green function. We should show that $\tilde
d$ is a distance, that it is equivalent to $d$, and hyperbolic, to be
able to apply the results of~\cite{coornaert} and~\cite{BHM:2}.

The subadditivity~\eqref{eq_Fr_subadd} of $F_R$ shows that $\tilde d$
satisfies the triangular inequality. For $x\not=y$, considering $n$
concatenations of paths from $x$ to $y$ then to $x$, one gets
  \begin{equation*}
  G_R(x,x) \geq \sum_{n=0}^\infty (F_R(x,y)F_R(y,x))^n.
  \end{equation*}
Since $G_R(x,x)$ is finite, this shows that $F_R(x,y)F_R(y,x)<1$.
Hence, $\tilde d$ is a distance. It is a variant of the Green
distance studied in~\cite{BHM:2}.

The quantity $G_R(e,e)$, which is finite, equals $\sum_{\gamma}
w_R(\gamma)$ (where the sum is over all paths from $e$ to itself, and
the notation $w_R(\gamma)$ for the $R$-weight of a path $\gamma$ has
been introduced in Subsection~\ref{subsec_green}). Excluding finitely
many paths, one can make the remaining sum arbitrarily small. If $x$
is not on one of those finitely many paths, then $F_R(e,x)F_R(x,e)$
is bounded by the remaining sum, and is therefore arbitrarily small.
This shows that $\tilde d(e,x)$ tends to infinity when $x\to \infty$
in $\Gamma$.

Since $G_R$ (or, equivalently, $F_R$) satisfies Ancona inequalities,
there exists $D>0$ such that $\tilde d(x,z) \geq \tilde d(x,y)+\tilde
d(y,z)-D$ whenever $x,y,z$ are on a geodesic segment in this order.
Let $L$ be such that $\tilde d(e,x)\geq 2D$ for $\lgth{x}\geq L$. By
induction, this implies that $\tilde d(e,x) \geq (n+1)D$ for
$\lgth{x}\geq nL$. In particular, there exists a constant $C>0$ such
that, for all $x$, $\tilde d(e,x) \geq C^{-1} \lgth{x}$. By Harnack
inequalities~\eqref{eq_harnack}, we also have $\tilde d(e,x) \leq
C\lgth{x}$. This shows that the distance $\tilde d$ is equivalent to
the word distance $d$.

The word distance is hyperbolic. It does not immediately follow that
$\tilde d$ is hyperbolic, since the metric space $(\Gamma, \tilde d)$
is usually not geodesic (while the invariance of hyperbolicity under
quasi-isometries requires such an assumption). However, \cite{BHM:2}
proves that if Ancona inequalities hold then $\tilde d$ is hyperbolic
(the proof given in their Theorem 1.1 is for the usual Green metric,
but it applies verbatim in our setting).

%On a heuristic level, the convergence of $\theta_s$ to $\nu_R$ when
%$s$ tends to $0$ translates into the fact that $\nu_R$ is the limit
%when $t$ tends to $1$ of
%  \begin{equation*}
%  \sum_{x\in \Gamma} e^{-t \tilde d(e,x)} \delta_x/ \sum_{x\in \Gamma} e^{-t\tilde d(e,x)}.
%  \end{equation*}
%This is not trivial to justify (since the factor $e^{-s\lgth{x}}$ in the
%definition of $\theta_s$ can not be replaced directly by $e^{-s
%\tilde d(e,x)}$), but we will not need it. What matters more is the
%intuition that $\nu_R$ is a Patterson-Sullivan measure for the
%hyperbolic distance $\tilde d$. The ergodicity of such measures in
%Gromov-hyperbolic spaces is known thanks to the work of Coornaert
%\cite{coornaert}. Since $(\Gamma, \tilde d)$ is not a geodesic space,
%these results to not apply directly, but they were extended to
%non-geodesic spaces in~\cite{BHM:2}. More
%precisely,

Finally, we can apply the results of Paragraphs~2.2 and~2.3
in~\cite{BHM:2}. The equation~\eqref{eq_quasiconformal} shows that
$\nu_R$ is quasi-conformal for a distance at infinity coming from the
hyperbolic distance $\tilde d$ on $\Gamma$. Therefore, \cite[Theorem
2.7]{BHM:2} implies that $\nu_R$ is ergodic.
\end{proof}

\begin{prop}
\label{prop_integral_equal}
For $i\not=i'$, one has $\int \phi_r\dd\mu_i =
\int\phi_r\dd\mu_{i'}$.
\end{prop}
\begin{proof}
The limit of $\boL_{\phi_R}^{np+q} f(\emptyset)$ is given by
$\sum_{i=1}^I \sum_{j=0}^{p_i-1}\left(\int f\dd\lambda_{i, (j-q \bmod
p_i)}\right) h_{i,j}(\emptyset)$. Since $h_{i,j}(\emptyset)$ is
bounded from above and from below, we deduce that $\nu_R$ is
equivalent to the push-forward under $\alpha_*$ of the measure
$\sum_{i,j} \lambda_{i,j}$ restricted to the set of paths beginning
with an edge in $E_*$.

The probability measure $\dd\mu_i = \frac{1}{p_i}\sum_{j=0}^{p_i-1}
h_{i,j} \dd\lambda_{i,j}$ is invariant and ergodic. Let $O_i \subset
\Sigma$ denote the set of points such that the normalized Birkhoff
sums $S_n f/n$ converge to $\int f\dd\mu_i$ for any continuous
function $f$. By Birkhoff ergodic theorem, $\mu_i(O_i)=1$. Since
$\mu_i$ is equivalent to $\beta_i = \sum_j \lambda_{i,j}$ restricted
to the set $\bSigma_i$ of paths staying in the component $\boC_i$, we
get $\beta_i(O_i^c \cap \bSigma_i)=0$ (where $O_i^c$ denotes the
complement of $O_i$). We deduce that
  \begin{equation}
  \label{eq_betai}
  \beta_i(O_i^c)=0.
  \end{equation}
Otherwise, since $\beta_i$-almost every point ends up in $\bSigma_i$
after finitely many iterations, we would have $\beta_i(O_i^c \cap
\sigma^{-k} \bSigma_i)>0$ for some $k\geq 0$, hence
$\beta_i(\sigma^{-k}(\sigma^k O_i^c \cap \bSigma_i))>0$. Since
$\sigma_*^k \beta_i$ is absolutely continuous with respect to
$\beta_i$ by Lemma~\ref{lem_absolute_continuity}, and $\sigma^k O_i^c
\subset O_i^c$, this gives $\beta_i(O_i^c \cap \bSigma_i)>0$, a
contradiction.

Let us now show that
  \begin{equation}
  \label{eq_nuR_pos}
  \nu_R( \alpha_*(O_i \cap [E_*])) > 0,
  \end{equation}
where $O_i\cap [E_*]$ denotes the set of paths in $O_i$ beginning
with an edge in $E_*$. Otherwise, since the image of $\beta_i(\cdot
\cap [E_*])$ is absolutely continuous with respect to $\nu_R$, we
would get $(\alpha_* \beta_i)(\alpha_*(O_i \cap [E_*])) = 0$, hence
$\beta_i(O_i \cap [E_*]) = 0$. Since $\beta_{i}(O_i^c)=0$
by~\eqref{eq_betai}, we get $\beta_i([E_*]) = 0$. This is a
contradiction since Theorem~\ref{thm_spectral_desc} shows that
$\beta_i$ gives positive weight to $[E_*]$.

Consider now $\omega\in O_i\cap [E_*]$, and let $\xi =
\alpha_*(\omega) \in \partial \Gamma$. The path $\alpha(\omega)$ is a
geodesic converging to $\xi$. In particular, denoting by
$\bar\omega_n$ the beginning of $\omega$ of length $n$,
$x_n=\alpha_*(\bar\omega_n)$ is a sequence of points converging to
$\xi$ along a geodesic ray. Moreover,
  \begin{equation*}
  \log H_r(e,x_n) = S_n \phi_r(\bar\omega_n) + \log H_r(e,e).
  \end{equation*}
Since $\phi_r$ is H\"older continuous, $S_n \phi_r(\bar\omega_n) - S_n
\phi_r(\omega)$ remains uniformly bounded. Hence, $\log H_r(e,x_n) /
n = S_n \phi_r(\omega)/n + o(1)$ tends to $c_i = \int \phi_r
\dd\mu_i$ by definition of $O_i$. This shows that $\xi \in U(c_i)$.
Therefore, $\alpha_*(O_i\cap [E_*]) \subset U(c_i)$.
With~\eqref{eq_nuR_pos}, this gives $\nu_R(U(c_i)) > 0$.

Since $\nu_R$ is ergodic for the action of $\Gamma$ by
Proposition~\ref{prop_ergodic}, and the sets $U(c)$ are
$\Gamma$-invariant, we deduce that $U(c_i)$ has full measure.
Therefore, all those sets have to coincide.
\end{proof}

\begin{proof}[Proof of Theorem~\ref{thm_Pressi_indep}]
By Proposition~\ref{prop_spectral_perturb}, the pressure
$\Press_i(\phi_r)$ on the component $\boC_i$ is equal to $\int
(\phi_r-\phi_R) \dd\mu_i + O(\norm{\phi_r-\phi_R})^2$. The integral
does not depend on $i$, by Proposition~\ref{prop_integral_equal}.
Considering $i'$ such that the pressure is maximal, we obtain
  \begin{equation*}
  \Press_i(\phi_r) = \Press(\phi_r) + O(\norm{\phi_r-\phi_R})^2.
  \end{equation*}
Moreover, by~\eqref{eq_bound_Hr_crude} and~\eqref{eq_sumHrasymp}, the
ratio between $\Press(\phi_r)$ and $-\sqrt{R-r}$ is bounded from
above and below. Since $\norm{\phi_r-\phi_R}^2 = O(R-r)^{2/3}$ by
Lemma~\ref{lem_perturb_holder}, we obtain $\Press_i(\phi_r) =
\Press(\phi_r) + o(\Press(\phi_r))$. This concludes the proof.
\end{proof}

\subsection{Estimating the second derivative of the Green function}

To improve on Proposition~\ref{prop_aprioridiff}, one should get
asymptotics for the function $\sum_{x,y} G_r(e,y) G_r(y,x)G_r(x,e)$
in terms of $\eta(r)=\sum G_r(e,x)G_r(x,e) = \sum H_r(e,x)$ or,
equivalently, in terms of $\Press(\phi_r)$.

\begin{prop}
\label{prop_triple_sum}
One has when $r\to R$
  \begin{equation}
  \label{eq_triplesum}
  \sum_{x,y} G_r(e,y) G_r(y,x)G_r(x,e) = c(r) \eta(r)^3 +
  O(\eta(r)^2),
  \end{equation}
for some nonnegative function $c(r)$ that extends continuously to
$r=R$.
\end{prop}

This subsection is devoted to the proof of this proposition. We will
need to express things in terms of transfer operators on the symbolic
space. Let
  \begin{equation}
  \label{eq_def_Phir}
  \Phi_r(x) = \frac{\sum_y G_r(e,y)G_r(y,x)}{G_r(e,x)},
  \end{equation}
and define for $r<R$ a probability measure $\nu_r$ on $\Gamma$ by
  \begin{equation}
  \label{eq_defnur}
  \int f\dd\nu_r = \frac{\sum_{x\in \Gamma} H_r(e,x) f(x)}{\sum_{x\in \Gamma} H_r(e,x)}.
  \end{equation}
The sum in~\eqref{eq_triplesum} is equal to
  \begin{equation*}
  \eta(r) \int \Phi_r \dd\nu_r.
  \end{equation*}
To estimate it, we should understand $\Phi_r$ and $\nu_r$.

\begin{prop}
\label{prop_nur_converges}
When $r\to R$, the sequence of probability measures $\nu_r$ on the
compact space $\Gamma\cup \partial \Gamma$ converges weakly to a
probability measure $\nu_R$, which is supported on $\partial \Gamma$.
\end{prop}
\begin{proof}
If $\nu_r$ converges weakly, then the limiting measure can give no
weight to $\Gamma$, since $\nu_r(x) = H_r(e,x)/\sum_{y\in \Gamma}
H_r(e,y)$ tends to $0$ by~\eqref{eq_bound_Hr_crude}.

Therefore, we just have to prove the convergence of $\nu_r(f)$ for
any continuous function, or even for $f$ in a dense set of functions.
We will consider those $f$ such that the function $\tilde f$ defined
on $\Sigma^*$ by $\tilde f(\omega) = f(\alpha_*(\omega))$ belongs to
$\boH^\hol$. For such a function, we have
  \begin{equation*}
  \sum_{x\in \Gamma} H_r(e,x) f(x) = H_r(e,e) \sum_{n\in \N} \boL_{\phi_r}^n
  (1_{[E_*]}\tilde f)(\emptyset).
  \end{equation*}
Using the spectral description of
Proposition~\ref{prop_spectral_perturb}, we deduce that this can be
written as
  \begin{equation*}
  \sum_{i=1}^I c_{\tilde f}(i,r)/\abs{\Press_i(\phi_r)} + O(1),
  \end{equation*}
as in~\eqref{eq_sumHrasymp}, for some functions $c_{\tilde f}(i,r)$
that extend continuously up to $r=R$. Therefore,
by~\eqref{eq_asymp_eta_Press},
  \begin{equation*}
  \nu_r(f) = \frac{\sum_{i=1}^I c_{\tilde f}(i,r)/\abs{\Press_i(\phi_r)} + O(1)}{\tau_2(r)/\abs{\Press(\phi_r)}+O(1)}.
  \end{equation*}
Since all the quantities $\abs{\Press_i(\phi_r)}$ are asymptotic to
$\abs{\Press(\phi_r)}$ by Theorem~\ref{thm_Pressi_indep} and tend to
$0$, this converges when $r$ tends to $R$ (to $\sum_{i=1}^I c_{\tilde
f}(i,R)/\tau_2(R)$).
\end{proof}
\begin{rmk}
One can easily check that the measure $\nu_R$ in
Proposition~\ref{prop_nur_converges} is the same as the measure we
constructed in Subsection~\ref{subsec_integ_indep} and was already
denoted by $\nu_R$. This will have no importance for our purposes.
\end{rmk}

To estimate $\Phi_r$ (defined in ~\eqref{eq_def_Phir}), let us first
note the following estimate.
\begin{lem}
\label{lem_Phir_basic}
We have
  \begin{equation*}
  \Phi_r(x) \leq C(1+\lgth{x}) \eta(r).
  \end{equation*}
\end{lem}
\begin{proof}
The proof relies on the same argument as Lemma~\ref{lem_C0bound}.
Denote by $\gamma$ a geodesic segment from $e$ to $x$, and by
$\Gamma_n$ (for $0\leq n\leq \lgth{x}$) the set of points whose first
projection on $\gamma$ is the point $\gamma(n)$, at distance $n$ of
$e$. For $y\in \Gamma_n$, one has by Ancona inequalities
  \begin{equation*}
  G_r(e,y) G_r(y,x) \leq C G_r(e,\gamma(n))G_r(\gamma(n), y) G_r(y,\gamma(n)) G_r(\gamma(n), x)
  \leq C H_r(\gamma(n), y) G_r(e,x).
  \end{equation*}
Therefore,
  \begin{equation*}
  \Phi_r(x) = \sum_{n=0}^{\lgth{x}} \sum_{y\in \Gamma_n} \frac{G_r(e,y) G_r(y,x)}{G_r(e,x)}
  \leq C \sum_{n=0}^{\lgth{x}} \sum_{y\in \Gamma_n} H_r(\gamma(n), y)
  \leq C(\lgth{x}+1)\eta(r).
  \qedhere
  \end{equation*}
\end{proof}

To obtain a convergence instead of bounds, we will use a similar
argument, but we will need to replace the wild sets $\Gamma_n$ by a
nicer version given by partitions of unity, as in lemma 8.5 of
\cite{gouezel_lalley} (that we recall for the convenience of the
reader):
\begin{lem}
For $K$ large enough, we can associate to any geodesic segment
$\gamma$ in the Cayley graph of length $2K+1$ centered around $e$ a
function $\kappa_\gamma: \Gamma \to [0,1]$ with the following
properties:
\begin{enumerate}
\item The function $\kappa_\gamma$ extends continuously to
    $\Gamma\cup
    \partial \Gamma$.
\item Let $\pi_\gamma(y)$ be the set of points on $\gamma$ that
    are closest to $y\in \Gamma$. Then $\kappa_\gamma(y)=0$ if
    $\pi_\gamma(y)$ contains a point at distance $\geq K/4$ of
    $e$.
\item Let $\gamma'$ be any biinfinite geodesic passing through
    $e$. Adding the functions $\kappa_\gamma$ along the
    subsegments of $\gamma'$ of length $2K+1$ one gets the
    function identically equal to $1$. More formally, for all
    $y\in \Gamma$,
  \begin{equation}
  \label{eq_kappa_somme}
  \sum_{n\in \Z} \kappa_{\gamma'(n)^{-1}\gamma'[n-K, n+K]}(\gamma'(n)^{-1}y)=1.
  \end{equation}
\end{enumerate}
\end{lem}

Let us now define for $r\in [1,R)$ a function $\Psi_r$ on geodesic
segments $\gamma$ through $e$, as follows. Let $a$ and $b$ be the
endpoints of $\gamma$. If $d(e,a)\leq K$ or $d(e,b)\leq K$, let
$\Psi_r(\gamma)=0$. Otherwise, let
  \begin{equation*}
  \Psi_r(\gamma)=\eta(r)^{-1}\sum_{y\in \Gamma}\kappa_{\gamma[-K, K]}(y)G_r(a,y)G_r(y,b)/G_r(a,b).
  \end{equation*}
Consider a geodesic segment $\gamma$ from $e$ to a point $x$, and
denote by $\sigma^n \gamma$ the shifted segment, i.e.,
$\gamma(n)^{-1} \gamma$. Then we have
  \begin{equation}
  \label{eq_xoiuwvpoiuxcv}
  \Phi_r(x) = \eta(r) \sum_{n=0}^{\lgth{x}} \Psi_r(\sigma^n \gamma) + O(\eta(r)).
  \end{equation}
Indeed, by~\eqref{eq_kappa_somme}, when one adds all the quantities
$\Psi_r(\sigma^n \gamma)$, one counts every point in the group with a
coefficient $1$, excepted those whose projection on $\gamma$ is close
to $e$ or $x$. They contribute to the sum by an amount at most
$C\eta(r)$, as explained in the proof of Lemma~\ref{lem_Phir_basic}.

\begin{lem}
\label{lem_Psir_holder}
The functions $\Psi_r$ are uniformly bounded and H\"older-continuous
for $r\in [1,R)$. They converge uniformly when $r$ tends to $R$.
\end{lem}
By H\"older continuous, we mean that, if two geodesics $\gamma$ and
$\gamma'$ coincide on a ball of size $n$ around $e$, then
$\abs{\Psi_r(\gamma)-\Psi_r(\gamma')} \leq Ce^{-\rho n}$ for some
$\rho>0$.
\begin{proof}
This is essentially Lemma~8.6 in \cite{gouezel_lalley}. The uniform
H\"older continuity is proved there and relies uniquely on strong
Ancona inequalities. On the other hand, the proof of the convergence
when $r\to R$ has to be modified slightly due to the presence of
several maximal components.

Since the functions $\Psi_r$ are uniformly H\"older continuous, it is
sufficient to show that they converge simply to get uniform
convergence. Fix some geodesic segment $\gamma$ through $e$, with
endpoints $a$ and $b$ at distance at least $K$ of $e$. We have
  \begin{equation*}
  G_r(a,b) \Psi_r(\gamma) = \frac{1}{\eta(r)}\sum_{y\in \Gamma}\kappa_{\gamma[-K, K]}(y)\frac{G_r(a,y)G_r(y,b)}{G_r(e,y)G_r(y,e)} H_r(e,y)
  = \int f_r(y) \dd\nu_r,
  \end{equation*}
where $f_r(y) = \kappa_{\gamma[-K,
K]}(y)\frac{G_r(a,y)G_r(y,b)}{G_r(e,y)G_r(y,e)}$. This is a function
on $\Gamma$ that extends continuously to $\partial \Gamma$ by the
strong Ancona inequalities (and since $\kappa_{\gamma[-K, K]}$ is
continuous). Moreover, $f_r$ converges uniformly to a function $f_R$
when $r$ tends to $R$, by Lemma~\ref{lem_C0bound}. Since $\nu_r$
converges weakly by Proposition~\ref{prop_nur_converges}, it follows
that $\int f_r \dd\nu_r$ converges.
\end{proof}

\begin{lem}
There exists a family of functions $h_r$ on $\bSigma$ for $r\in
[1,R)$ with the following properties:
\begin{enumerate}
\item The functions $h_r$ are H\"older continuous, and they
    converge in the H\"older topology to a function $h_R$ when
    $r\to R$.
\item For any $\omega\in \Sigma^*$ of length $n$,
  \begin{equation}
  \label{eq_PhirasSn}
  \Phi_r(\alpha_*(\omega))
  =\eta(r) S_n h_r(\omega) + O(\eta(r)).
  \end{equation}
\end{enumerate}
\end{lem}
We recall that $S_n h_r$ is the Birkhoff sum
$\sum_{k=0}^{n-1}h_r\circ \sigma^k$.
\begin{proof}
We will need to work with the bilateral shift $\sigma_{\Z}$ on the
space $\bSigma_{\Z}$ of bilateral paths in the automaton (that may be
infinite in zero, one or both directions). We define a function $g_r$
on the set $\Sigma^*_\Z$ of finite paths by $g_r(\omega)=0$ if the
length of $\omega$ in the future or in the past is less than $K$, and
$g_r(\omega) = \Psi_r(\alpha(\omega))$ otherwise, where
$\alpha(\omega)$ is the geodesic segment going through
  \begin{equation*}
  \dotsc,  \alpha(\omega_{-1})^{-1}\alpha(\omega_{-2})^{-1},
  \alpha(\omega_{-1})^{-1}, e, \alpha(\omega_0), \alpha(\omega_0)\alpha(\omega_1), \dotsc.
  \end{equation*}
Lemma~\ref{lem_Psir_holder} ensures that the functions $g_r$ are
H\"older continuous, and that they converge uniformly when $r$ tends to
$R$. By~\eqref{eq_relate_holder}, they also converge in some H\"older
topology. Moreover, they extend to H\"older continuous functions on
$\bSigma_{\Z}$.

Consider now a finite path $\omega$ in $\Sigma^*$, of length $n$. One
may consider it as a path in $\Sigma^*_{\Z}$ with empty coordinates
for negative time. The equation~\eqref{eq_xoiuwvpoiuxcv} reads
  \begin{equation*}
  \Phi_r(\alpha_*(\omega)) = \eta(r)\sum_{k=0}^n g_r(\sigma_{\Z}^k \omega) + O(\eta(r)).
  \end{equation*}

This is almost the required property, but the function $g_r$ is
defined on the bilateral shift instead of the unilateral shift as
desired. This problem is solved using a classical coboundary trick:
for any H\"older continuous function $g$ on $\bSigma_{\Z}$, there exist
two H\"older continuous functions $h$ and $u$ on $\bSigma_{\Z}$ (for a
smaller H\"older exponent) such that $g=h+u-u\circ \sigma_{\Z}$, and
$h$ only depends on positive coordinates. Moreover, $h$ and $u$
depend linearly (and continuously) on $g$. This is Proposition 1.2 in
\cite{parry-pollicott}. The proof is given there for subshifts where
one only allows infinite paths, but it readily adapts to the
situation where finite paths are allowed (or one can reduce to the
infinite paths situation by adding two cemeteries, one for the past
and one for the future).

Writing $g_r=h_r+u_r-u_r\circ \sigma_{\Z}$ as above, we obtain
  \begin{equation*}
  \Phi_r(\alpha_*(\omega))
  = \eta(r)\sum_{k=0}^n h_r(\sigma_{\Z}^k \omega)
    + \eta(r) (u_r(\omega) - u_r(\sigma_{\Z}^{n+1}\omega)) + O(\eta(r)).
  \end{equation*}
Since $u_r$ is uniformly bounded, this is the desired decomposition.
\end{proof}

\begin{proof}[Proof of Proposition~\ref{prop_triple_sum}]
The sum in~\eqref{eq_triplesum} can be written as $\sum_{x\in \Gamma}
H_r(e,x) \Phi_r(x)$. Since $\alpha_*$ induces a bijection between the
finite paths in the automaton starting from $s_*$ and the group, this
is equal to $\sum H_r(e, \alpha_*(\omega)) \Phi_r(\alpha_*(\omega))$,
where the sum is restricted to those paths with $\omega_0\in E_*$.

Consider now, in this sum, the contribution of paths of length $n$.
By definition of the transfer operator $\boL_{\phi_r}$, it is equal
to $H_r(e,e) \boL_{\phi_r}^n (1_{[E_*]} \cdot \Phi_r\circ
\alpha_*)(\emptyset)$. Using the decomposition~\eqref{eq_PhirasSn}
for $\Phi_r\circ \alpha_*$, we get
  \begin{equation*}
  \sum_{x,y\in \Gamma} G_r(e,y)G_r(y,x)G_r(x,e) = \eta(r) \sum_{n\in \N}\boL_{\phi_r}^n (1_{[E_*]}S_n u_r + O(1))(\emptyset).
  \end{equation*}
The contribution of the error term $O(1)$ in this equation is bounded
by $\eta(r) \sum \norm{\boL_{\phi_r}^n 1} \leq C
\eta(r)/\abs{\Press(\phi_r)} \leq C \eta(r)^2$. It is therefore
compatible with the error term in the statement of
Proposition~\ref{prop_triple_sum}.

Since $\boL_{\phi_r}(u\cdot v\circ \sigma)=v \boL_{\phi_r} u$, we
have $\boL_{\phi_r}^n (1_{[E_*]}S_n u_r) = \sum_{k=1}^{n}
\boL_{\phi_r}^k (u_r \boL_{\phi_r}^{n-k} 1_{[E_*]})$. Therefore, the
previous equation becomes
  \begin{equation*}
  \eta(r) \sum_{n=0}^\infty \sum_{k=1}^n \boL_{\phi_r}^k (u_r \boL_{\phi_r}^{n-k} 1_{[E_*]})(\emptyset) + O(\eta(r)^2)
  = \eta(r) \sum_{k=1}^\infty \boL_{\phi_r}^k\left(u_r \sum_{\ell=0}^\infty \boL_{\phi_r}^\ell 1_{[E_*]}\right)(\emptyset) + O(\eta(r)^2).
  \end{equation*}
Using the spectral description of
Proposition~\ref{prop_spectral_perturb}, one can write
  \begin{equation*}
  \sum_{\ell=0}^\infty \boL_{\phi_r}^\ell 1_{[E_*]} = \sum_{i=1}^I
  f_{i,r} / \abs{\Press_i(\phi_r)} + O(1),
  \end{equation*}
for some H\"older continuous functions $f_{i,r}$ that converge when $r$
tends to $R$. Again, the $O(1)$ results in an error $O(\eta(r)^2)$ in
the final formula. It remains to understand $\sum_{k=1}^\infty
\boL_{\phi_r}^k(u_r f_{i,r})$. Using again the spectral description
of $\boL_{\phi_r}$, one may write it as $\sum_{j=1}^I
g_{i,j,r}/\abs{\Press_j(\phi_r)}$ for some H\"older continuous
functions $g_{i,j,r}$ that depend continuously on $r$. Finally, we
have obtained
  \begin{equation*}
  \sum_{x,y\in \Gamma} G_r(e,y)G_r(y,x)G_r(x,e) = \eta(r) \sum_{i=1}^I \sum_{j=1}^I
    \frac{g_{i,j,r}(\emptyset)}{\abs{\Press_i(\phi_r)} \abs{\Press_j(\phi_r)}}
  +O(\eta(r)^2).
  \end{equation*}
By Theorem~\ref{thm_Pressi_indep} and~\eqref{eq_asymp_eta_Press},
$\eta(r)$ is asymptotic to $\tau_2(r)/\abs{\Press_i(\phi_r)}$ for
some continuous function $\tau_2$ that admits a positive limit at
$r=R$. Since all the quantities $g_{i,j,r}(\emptyset)$ converge when
$r$ tends to $R$, the proposition follows.
\end{proof}

\subsection{Asymptotics of the Green function}

In this subsection, we prove Theorem~\ref{thm_strong_asymp}. We rely
on the asymptotics for $\sum G_r(e,y)G_r(y,x)G_r(x,e)$ that were
obtained in Proposition~\ref{prop_triple_sum} and the differential
equation for $G_r(e,e)$.

More precisely, define as in the proof of
Corollary~\ref{cor_crude_eta} a function $F(r) = r^2 \eta(r)$. It
satisfies $F'(r) = 2r \sum_{x,y} G_r(e,y) G_r(y,x)G_r(x,e)$. By
Proposition~\ref{prop_triple_sum}, $2F'(r)/F(r)^3$ converges to a
constant $c$ when $r$ tends to $R$. By
Proposition~\ref{prop_aprioridiff}, $c$ is nonzero. By integration,
it follows that $1/F(r)^2 - 1/F(R)^2 \sim c(R-r)$. Since
$F(R)=+\infty$, we get $F(r) \sim c^{-1/2}(R-r)^{-1/2}$. This proves
the desired asymptotics of $\partial G_r(e,e)/\partial r$.

Fix now a point $a\in\Gamma$, let us compute the asymptotics of
$\partial G_r(e,a)/\partial r$ or, equivalently, of $\sum_x G_r(e,x)
G_r(x,a)$. We write this sum as
  \begin{equation*}
  \sum_x H_r(e,x) \frac{G_r(x,a)}{G_r(x,e)}
  =\eta(r) \int_\Gamma f_r \dd\nu_r,
  \end{equation*}
where $f_r(x) = G_r(x,a)/G_r(x,e)$ and the measure $\nu_r$ has been
defined in~\eqref{eq_defnur}. The functions $f_r$ extend continuously
to $\Gamma\cup \partial \Gamma$ by the strong Ancona inequalities,
and converge uniformly when $r$ tends to $R$, by
Lemma~\ref{lem_C0bound}. Since $\nu_r$ converges weakly when $r\to R$
by Proposition~\ref{prop_nur_converges}, we deduce that $\int
f_r\dd\nu_r$ converges. Therefore, the asymptotics of $\partial
G_r(e,a)/\partial r$ follow from those of $\eta(r)$. \qed

\section{Asymptotics of transition probabilities}
\label{sec_tauberian}

Theorem~\ref{main_thm} follows directly from the asymptotics of the
Green function proved in Theorem~\ref{thm_strong_asymp} and from
Theorem~9.1 in \cite{gouezel_lalley}: this theorem shows that, for
\emph{symmetric} measures, one can read the behavior of transition
probabilities from the behavior of the Green function.

Let us explain quickly why symmetry matters. The Green function is
$\sum r^n p_n(x,y)$, its derivative is $\sum n r^{n-1} p_n(x,y)$. If
$\partial G_r(x,y)/\partial r \sim C (R-r)^{-1/2}$, it follows from
Karamata's tauberian theorem that
  \begin{equation*}
  \sum_{k=1}^n k R^k p_k(x,y) \sim C' n^{1/2}.
  \end{equation*}
This is a local limit theorem in Cesaro average. If $R^n p_n(x,y)$
were monotone, the desired asymptotics of $p_n(x,y)$ would follow
readily. Symmetry is used to obtain almost monotonicity: up to an
exponentially small error (that does not matter in the estimates),
$R^n p_n(x,y)$ is indeed decreasing when the random walk is aperiodic
and the measure is symmetric. This is a consequence of spectral
results for the (self-adjoint) Markov operator associated to the
random walk.

From Theorem~\ref{main_thm}, one can also derive asymptotics for the
first return probabilities. We describe the result in the aperiodic
case, the periodic one is handled similarly by looking at $\mu^2$.
\begin{prop}
\label{prop_first_return_estimates}
Consider a probability measure $\mu$ on a countable group $\Gamma$
such that the associated transition probabilities satisfy $p_n(x,y)
\sim C(x,y) R^{-n} n^{-\beta}$ for some $R\geq 1$ and $\beta>1$. Let
$f_n(x,y)$ be the first visit probabilities from $x$ to $y$ at time
$n$, i.e.,
  \begin{equation*}
  f_n(x,y) = \Pbb_x(X_1,\dotsc, X_{n-1}\not=y, X_n = y).
  \end{equation*}
Then $f_n(x,y) \sim C'(x,y) R^{-n} n^{-\beta}$ for some constants
$C'(x,y)$.
\end{prop}
For the proof, we will mainly rely on the following theorem
(\cite[Theorem 1]{chover_banach_alg}):
\begin{thm}
\label{thm:banach_alg}
Consider a function $A(z) = \sum_{n=0}^\infty a_n z^n$ with $a_n\geq
0$ and $a_n\sim c R^{-n} n^{-\beta}$ with $\beta>1$. Consider also a
function $\Phi$ which is analytic on a neighborhood of $\{f(z) \st
|z|\leq R\}$. Then the coefficients $b_n$ of the series expansion
$\Phi(A(z)) = \sum b_n z^n$ satisfy
  \begin{equation*}
  b_n \sim a_n \Phi'(\sum a_n R^n).
  \end{equation*}
\end{thm}
\begin{proof}[Proof of Proposition~\ref{prop_first_return_estimates}]
Decomposing a path from $e$ to itself into successive excursions, one
gets the renewal equation
  \begin{equation}
  \label{eq_renewal}
  \sum_{n=0}^\infty p_n(e,e)z^n= \frac{1}{1-\sum_{n=1}^\infty f_n(e,e) z^n}.
  \end{equation}
Since $\sum p_n(e,e)R^n =G_R(e,e)<\infty$, one deduces $\sum f_n(e,e)
R^n < 1$. Therefore, $1-\sum f_n(e,e) z^n$ does not vanish for
$|z|\leq R$, and $\sum p_n(e,e)z^n$ is well defined and nonzero for
any such $z$. From~\eqref{eq_renewal}, one gets
  \begin{equation*}
  \sum f_n(e,e)z^n=1-1/\sum p_n(e,e)z^n=\Phi(\sum p_n(e,e)z^n),
  \end{equation*}
where we set $\Phi(t)=1-1/t$. Since $p_n(e,e) \sim C
R^{-n}n^{-\beta}$, we may apply Theorem~\ref{thm:banach_alg} to
obtain $f_n(e,e) \sim C' R^{-n} n^{-\beta}$.

For $x\not =y$, one has $\sum p_n(x,y) z^n = \left(\sum
f_n(x,y)z^n\right)\cdot \left(\sum p_n(e,e)z^n\right)$. The functions
$\sum p_n(x,y)z^n$ and $1/\sum p_n(e,e)z^n$ both have coefficients
that are asymptotic to a constant times $R^{-n}n^{-\beta}$. Since the
set of all functions with this property is closed under
multiplication (see~\cite[Lemma 1]{chover_banach_alg}), we get
$f_n(x,y)\sim C'(x,y) R^{-n}n^{-\beta}$ as desired.
\end{proof}

\appendix

\section{Ancona inequalities for surface groups}
\label{sec_appendix}

In this appendix, we prove Ancona inequalities for surface groups
without any symmetry assumption on the measure:

\begin{thm}
\label{thm_main_SL}
Let $\Gamma$ be a cocompact Fuchsian group, i.e., a cocompact
discrete subgroup of $\SL$. Let $\mu$ be an admissible finitely
supported probability measure on $\Gamma$. Then it satisfies strong
uniform Ancona inequalities.
\end{thm}

This theorem has several corollaries:
\begin{cor}
Under the assumptions of the theorem, the Martin boundary for
$R$-harmonic functions coincides with the geometric boundary of the
group, i.e., the unit circle $S^1$.
\end{cor}

\begin{cor}
Under the assumptions of the theorem, for any $x,y\in \Gamma$, there
exists $C(x,y)>0$ such that $\sum_{k=1}^n k R^k p_k(x,y) \sim C(x,y)
n^{1/2}$.
\end{cor}

The first corollary is a classical consequence of Ancona
inequalities. For the second corollary, we rely on
Section~\ref{sec_diff_eq} (or on~\cite{gouezel_lalley}, since the
Cannon automaton is transitive) to deduce that the Green function
satisfies $\partial G_r(x,y)/\partial r \sim C(x,y) /\sqrt{R-r}$ when
$r$ tends to $R$. Using Karamata's tauberian theorem, this readily
implies the statement of the corollary (see the arguments in
Section~\ref{sec_tauberian}). Note that we are unable to deduce the
true local limit theorem from this estimate since we do not know if
$R^n p_n(x,y)$ is decreasing, or sufficiently well approximated by a
decreasing sequence, as in the symmetric situation.

The strong uniform Ancona inequalities of Theorem~\ref{thm_main_SL}
are a consequence of estimates for the weight of paths avoiding a
ball, given in the following proposition, and of
Lemma~\ref{lem_abstract_ancona}.

\begin{prop}
\label{prop_dim2}
Let $\Gamma$ be a cocompact Fuchsian group in $\SL$. Let $\mu$ be an
admissible finitely supported probability measure on $\Gamma$. For
any $K>0$, there exists $n_0>0$ such that, for any $n\geq n_0$, for
any points $x,y,z$ on a geodesic segment (in this order) with
$d(x,y)\in [n,100n]$ and $d(y,z) \in [n,100 n]$,
  \begin{equation*}
  G_R(x,z; B(y,n)^c) \leq K^{-n}.
  \end{equation*}
\end{prop}

The rest of this section is devoted to the proof of the proposition.
The main tool in this proof is superadditivity (as in
Lemma~\ref{lem_avoidballs}, that relies on
Lemma~\ref{lem_GRbounded}). This time, it is in the form of Kingman's
subadditive ergodic theorem, or rather a bilateral version of this
theorem that we now give.

\begin{thm}
\label{thm_kingman}
Let $T:\Omega \to \Omega$ be an ergodic probability preserving
invertible automorphism of a probability space $(\Omega, \Pbb)$.
Consider for each bounded interval $I\subset \Z$ an integrable
function $\Phi_I:\Omega \to \R$ with the following properties:
\begin{enumerate}
\item If $I$ is the disjoint union of two intervals $I_1$ and
    $I_2$, then $\Phi_I(\omega) \leq \Phi_{I_1}(\omega) +
    \Phi_{I_2}(\omega)$.
\item One has $\Phi_{[m,n]}(\omega) = \Phi_{[m-1,
    n-1]}(T\omega)$.
\item The quantity $n^{-1} \int \Phi_{[0, n)}(\omega) \dd
    \Pbb(\omega)$ is bounded from below.
\end{enumerate}
Then, for almost every $\omega$, the quantity $(m+n)^{-1} \Phi_{[-m,
n)}(\omega)$ converges when $m+n\to \infty$ (and $m,n\geq 0$) towards
$\inf n^{-1} \int \Phi_{[0, n)}(\omega) \dd \Pbb(\omega) = \lim
n^{-1} \int \Phi_{[0, n)}(\omega) \dd \Pbb(\omega)$.
\end{thm}
\begin{proof}
Let $\Psi_n(\omega) = \Phi_{[0,n)}(\omega)$. The assumptions give,
for any $m,n\geq 0$,
  \begin{align*}
  \Psi_{m+n}(\omega) & = \Phi_{[0,m+n)}(\omega)
  \leq \Phi_{[0,n)}(\omega) + \Phi_{[n, n+m)}(\omega)
  = \Phi_{[0,n)}(\omega) + \Phi_{[0, m)}(T^n \omega)
  \\& = \Psi_n(\omega) + \Psi_m(T^n \omega).
  \end{align*}
This shows that $\Psi_n$ is a subadditive cocycle in the usual sense
of Kingman's ergodic theorem (see for instance \cite[Theorem
I.5.3]{krengel_ergodic_theorems}). Therefore, $n^{-1}\Psi_n$
converges almost surely and in $L^1$ to the limit $c=\inf n^{-1} \int
\Phi_{[0, n)}(\omega) \dd \Pbb(\omega) = \lim n^{-1} \int \Phi_{[0,
n)}(\omega) \dd \Pbb(\omega)$.

In the same way, $\Phi_{[-n,-1]}(\omega)$ is a subadditive cocycle
for the transformation $T^{-1}$. Hence, $n^{-1} \Phi_{[-n,
-1]}(\omega)$ converges almost surely to $\lim n^{-1} \int \Phi_{[-n,
-1]}(\omega) \dd \Pbb(\omega)$. Changing variables by
$\omega'=T^{-n}\omega$, this integral is equal to $\int \Phi_{[0,
n)}(\omega') \dd\Pbb(\omega')$. Therefore, the limit is again $c$.

Consider now a generic point $\omega$, and $m, n\geq 0$. We want to
show that if $m+n$ is large then $(m+n)^{-1}\Phi_{[-m, n)}(\omega)$
is close to $c$. We will do so if $m$ is large, the case $n$ large is
handled similarly. Let $\epsilon>0$ be small. Since
$j^{-1}\Phi_{[0,j)}(\omega')$ converges almost everywhere to $c$, it
converges uniformly on a set $A$ of measure arbitrarily close to $1$.
In particular, there exists $N$ such that, for all $j\geq N$ and for
all $\omega'\in A$, one has $j^{-1} \Phi_{[0,j)}(\omega') \in
[c-\epsilon, c+\epsilon]$. Since $\omega$ is generic and $\Pbb(A)$ is
very close to $1$, the orbit of $\omega$ spends a very large
proportion of its time in $A$. In particular, one may find for every
large enough $m$ an integer $k \in [m+\epsilon m, m+2\epsilon m]$
such that $T^{-k} \omega \in A$. We get for this $k$ (and for any
$n\geq 0$)
  \begin{align*}
  \Phi_{[0, k+n)}(T^{-k} \omega)
  \leq \Phi_{[0, k-m)}(T^{-k}\omega) + \Phi_{[k-m, k+n)}(T^{-k} \omega)
  = \Phi_{[0, k-m)}(T^{-k}\omega) + \Phi_{[-m, n)}(\omega).
  \end{align*}
If $m$ is large enough, then $k-m \geq \epsilon m$ is larger than
$N$. Since $T^{-k}\omega$ belongs to $A$, we obtain
  \begin{equation*}
  \Phi_{[-m, n)}(\omega) \geq (k+n)(c-\epsilon) - (k-m) (c+\epsilon)
  =(m+n)(c+O(\epsilon)).
  \end{equation*}
This shows that $\liminf (m+n)^{-1} \Phi_{[-m, n)}(\omega) \geq c$.
On the other hand,
  \begin{equation*}
  \Phi_{[-m, n)}(\omega) \leq \Phi_{[-m, -1]}(\omega) + \Phi_{[0,n)}(\omega)
  = m(c+o(1)) + n(c+o(1)).
  \end{equation*}
Therefore, $\limsup (m+n)^{-1} \Phi_{[-m, n)}(\omega) \leq c$. This
concludes the proof.
\end{proof}

Let us now start the proof of Proposition~\ref{prop_dim2}. We can
assume without loss of generality that $y=e$. Let $\ell>0$ be large
(it will not depend on $n$), we will construct $\ell$ suitable
barriers $A_1, \dotsc, A_\ell$ between $x$ and $z$ such that
  \begin{equation}
  \label{eq_dim2_barriers1}
  \sum_{\mathclap{a\in A_i, b\in A_{i+1}}} G_R(a,b)^2 \leq e^{-\rho n}
  \end{equation}
and
  \begin{equation}
  \label{eq_dim2_barriers2}
  \sum_{a\in A_1} G_R(e,a)^2 \leq 1,\quad \sum_{a\in A_\ell} G_R(a,e)^2 \leq 1
  \end{equation}
for some $\rho>0$ that does not depend on $\ell$, $x$ or $z$. From
the last equation, we obtain $\sum_{a\in A_1} G_R(x,a)^2 \leq C^n$
thanks to Harnack inequalities (and since $d(x,e) \leq 100 n$), and
$\sum_{a\in A_\ell} G_R(a, z)^2 \leq C^n$. Arguing as in the
beginning of the proof of Lemma~\ref{lem_avoidballs}, we define
operators $L_i$ from $\ell^2(A_{i+1})$ to $\ell^2(A_i)$. They satisfy
$\norm{L_0}\leq C^{n/2}$, $\norm{L_{\ell+1}} \leq C^{n/2}$ and
$\norm{L_i} \leq e^{-\rho n/2}$ for $1\leq i\leq \ell$. Therefore,
  \begin{equation*}
  G_R(x,z; B(e,n)^c) \leq \prod \norm{L_i} \leq C^n e^{-(\ell-1)\rho n/2}.
  \end{equation*}
Taking $\ell$ large, we can ensure that this is bounded by $K^{-n}$
as desired, for any $K>0$.

The key point of the argument is the construction of the barriers.
The problem with the argument in Lemma~\ref{lem_avoidballs} is that
we only have a control on $G_R(a,b)G_R(b,a)$ coming from
Lemma~\ref{lem_GRbounded}, not $G_R(a,b)^2$. The idea is that those
controls would be equivalent if $G_R(a,b)$ and $G_R(b,a)$ were of the
same order of magnitude. For symmetric measures, this is always the
case. For non-symmetric measures, we will be able to enforce it by
constructing the barriers using another, \emph{symmetric}, random
walk, and use Kingman subadditive ergodic theorem to show that for
typical points both $G_R(a,b)$ and $G_R(b,a)$ grow at the same speed.
It will then follow from Lemma~\ref{lem_GRbounded} that they are both
exponentially small.

Let us stress that this kind of argument can not work for \emph{all}
points. For instance, consider in the free group on two generators
$a$ and $b$ a random walk that goes towards $a$ with probability
$1-3\epsilon$, and towards $a^{-1}$, $b$ and $b^{-1}$ with
probability $\epsilon$, for some small enough $\epsilon$. It is easy
to check that $G_R(e,a^n)$ is exponentially large (while $G_R(a^n,
e)$ is exponentially small). In particular, $\sum_{x\in \Sbb_n}
G_R(e,x)^2$ grows exponentially fast, but this growth is due to a
rather small number of points. The barriers we construct have to
avoid those points.

We turn to details. The barriers we will construct will not depend on
the points $x$ and $z$ (but the order in which they will be
encountered will depend on those points, of course). Since $\Gamma$
is a cocompact discrete subgroup of $\SL$, it acts on the hyperbolic
disk $\HH^2$. If $O$ is a suitably chosen reference point in this
disk, the points $\gamma O$ (for $\gamma \in \Gamma$) are pairwise
disjoint, hence $\Gamma$ can be identified with $\Gamma O$. Moreover,
this identification is a quasi-isometry between $\Gamma$ (with the
word distance coming from its Cayley graph) and $\HH^2$. In
particular, the geometric boundary of $\Gamma$ is identified with
$S^1 = \partial \HH^2$. We can assume that $O$ is the center of the
hyperbolic disk.

Let us fix an admissible \emph{symmetric} measure $\nu$ on $\Gamma$,
supported on the set of generators, and let us consider the
corresponding random walk. We claim that the following lemma holds.
Here and henceforth, $G_R$ always denotes the Green function
associated to the original measure $\mu$.
\begin{lem}
\label{lem_dim2_barriersprop}
There exist $\rho>0$ and $v>0$ with the following properties.
\begin{enumerate}
\item For almost every trajectory $X_k$ of the random walk given
    by $\nu$, $d(X_k, e) \sim kv$.
\item For almost every trajectory $X_k$, for all large enough
    $k$, $G_R(e,X_k) \leq e^{-\rho k}$ and $G_R(X_k, e) \leq
    e^{-\rho k}$.
\item For almost every pair of independent trajectories $X_k$ and
    $Y_k$, for all large enough $k$ and $k'$, $G_R(X_k, Y_{k'})
    \leq e^{-\rho (k+k')}$.
\end{enumerate}
\end{lem}
Let us admit the lemma for the moment. We choose $2\ell+2$ points in
$S^1$ that are evenly spaced, and $2\ell+2$ small intervals $I_i$
around those points. The Poisson boundary of the random walk given by
$\nu$ is $S^1$, and the hitting measure has full support. Therefore,
there is positive probability to hit the boundary in any of the
intervals $I_i$. Let us choose for each $i$ a trajectory $X_k^{(i)}$
of the random walk that ends up in $I_i$. We will also require each
of those trajectories to be typical, so that they satisfy the
conclusions of Lemma~\ref{lem_dim2_barriersprop}.

Since the trajectories $X_k^{(i)}$ converge to different points on
the boundaries, they are disjoint outside of a large enough compact
set. Let $\Gamma_i(n)$ be the set of points $X_k^{(i)}$ that are at
distance at least $n$ of $e$, and let $B_i(n)$ be a thickening of
$\Gamma_i(n)$, i.e., $B_i(n)=\bigcup_{a\in \Gamma_i(n)} B(a, C_0)$
for some large constant $C_0$. If $n$ is large enough, the sets
$(B_i(n))_{i\leq 2\ell+2}$ are mutually disjoint. Since $d(X_k^{(i)},
e) \sim vk$, the set $\Gamma_i(n)$ only contains points among the
$X_k^{(i)}$ with $k\geq n/(2v)$. Therefore,
  \begin{equation*}
  \sum_{a\in \Gamma_i(n)} G_R(e,a)^2 \leq \sum_{k=n/(2v)}^\infty G_R(e, X_k^{(i)})^2
  \leq \sum_{k=v/(2v)}^\infty e^{-2\rho k} \leq C e^{-\rho n/v}.
  \end{equation*}
Since any point in the thickening $B_i(n)$ is a bounded distance away
from a point in $\Gamma_i(n)$, a similar estimate holds with $B_i(n)$
instead of $\Gamma_i(n)$ thanks to Harnack inequalities. Arguing in
the same way for the other inequalities, we obtain
  \begin{equation}
  \label{eq_estimate_barriers_dim2}
  \begin{split}
  &\sum_{a\in B_i(n)} G_R(e,a)^2 \leq C e^{-\rho n/v},
  \quad
  \sum_{a\in B_i(n)} G_R(a,e)^2 \leq C e^{-\rho n/v},
  \\
  &\sum_{a\in B_i(n), b\in B_j(n)}
  G_R(a,b)^2 \leq C e^{-2\rho n/v}.
  \end{split}
  \end{equation}

Consider now two points $x$ and $z$ at distance at least $n$ of $e$
such that $e$ is on a geodesic segment from $x$ to $z$. The Cayley
geodesic from $x$ to $z$ is a quasi-geodesic in hyperbolic space,
that remains in a bounded size neighborhood of a true hyperbolic
geodesic (and this geodesic passes close to $O$). Avoiding a ball
around $O$ in $\HH^2$, one can go from $x$ to $z$ in two directions
around this ball, clockwise or counterclockwise. Since the limit
intervals $I_i$ are evenly spaced, it follows that, in any of those
directions, one meets successively at least $\ell$ sets $B_i(n)$
(discarding if necessary the two sets that contain $x$ and $z$).
Denote by $A_1$ the union of the two sets $B_i(n)$ that are closest
to $x$ (ignoring the single set that might contain $x$), then $A_2$
the union of the next two ones, and so on. We get barriers
$A_1,\dotsc,A_\ell$ between $x$ and $z$ as desired. Moreover, the
estimates~\eqref{eq_estimate_barriers_dim2} show that those barriers
satisfy~\eqref{eq_dim2_barriers1} and ~\eqref{eq_dim2_barriers2} if
$n$ is large enough (for a different value of $\rho$). This concludes
the proof of Proposition~\ref{prop_dim2}, modulo
Lemma~\ref{lem_dim2_barriersprop}.\qed

\begin{proof}[Proof of Lemma~\ref{lem_dim2_barriersprop}]
We will use Kingman's theorem on the space $\Omega = \Gamma^{\Z}$
with the product measure $\nu^{\otimes \Z}$. In other words, an
element $\omega \in \Omega$ is a sequence of elements $\omega_i$ of
$\Gamma$ that are drawn independently according to $\nu$. Such an
$\omega$ can be viewed as the increments of a random walk distributed
according to $\nu$: let $X_n(\omega) = \omega_0 \dotsm \omega_{n-1}$
for $n\geq 0$ and $X_n(\omega) = \omega_{-1}^{-1} \dotsm
\omega_{n}^{-1}$ for $n<0$, so that $X_0=e$ and $X_{n+1}=X_n
\omega_{n}$. Let $T$ be the left shift on $\Omega$, it is ergodic,
preserves the measure, and $X_n(T\omega) = \omega_0^{-1}
X_{n+1}(\omega)$.

We define a subadditive cocycle $\Phi_{[m,n)}(\omega) = -\log F_R(
X_m(\omega), X_n(\omega))$, where $F_R(x,y)=G_R(x,y)/G_R(e,e)$ is the
first entrance Green function defined in
Subsection~\ref{subsec_green}. This function satisfies
$F_R(x,y)F_R(y,z) \leq F_R(x,z)$ by~\eqref{eq_Fr_subadd}, hence
$\Phi_{[m,n)}$ is subadditive. By Harnack inequalities,
$\abs{\Phi_{[m,n)}(\omega)} \leq C\abs{n-m}$. Therefore, the
integrability assumptions of Theorem~\ref{thm_kingman} are satisfied.
We deduce that $(m+n)^{-1}\log F_R(X_{-m}, X_n)$ converges almost
surely to $c = \lim n^{-1} \int \log F_R(e, X_n(\omega))
\dd\Pbb(\omega)$. We can also apply Kingman's theorem to the cocycle
$\tilde \Phi_{[m,n)}(\omega) = -\log F_R( X_n(\omega), X_m(\omega))$,
to get that $(m+n)^{-1}\log F_R(X_{n}, X_{-m})$ almost surely
converges, to $c'=\lim n^{-1} \int \log F_R(X_n(\omega), e)
\dd\Pbb(\omega)$. We have $F_R(X_n, e) = F_R(e, X_n^{-1})$. Since
$X_n^{-1}$ is distributed like $X_n$ by symmetry of the random walk,
this gives $\int \log F_R(X_n(\omega), e) \dd\Pbb(\omega) = \int \log
F_R(e, X_n(\omega)) \dd\Pbb(\omega)$. Dividing by $n$ and letting $n$
tends to infinity gives $c=c'$.

The quantities $(m+n)^{-1}\log F_R(X_{-m},X_n)$ and $(m+n)^{-1} \log
F_R(X_n, X_{-m})$ both converge to $c$. Taking $m=0$, we get in
particular that $n^{-1}\log F_R(e,X_n)$ and $n^{-1}\log F_R(X_n, e)$
almost surely converge to $c$.

We will now prove that $c$ is strictly negative. There exists a real
number $h\geq 0$ (the entropy of the random walk) such that the
random walk at time $n$ is essentially supported by $e^{h n}$ points,
with a probability $e^{-hn}$ to reach each of those points. More
precisely (see for instance \cite[Theorem 2.28]{furman}), for any
$\epsilon>0$, if $n$ is large enough, there exists a subset $E_n$ of
$\Omega$, with probability at least $3/4$, such that for any
$\omega\in E_n$ one has
  \begin{equation*}
  \Pbb\{\omega': X_n(\omega') = X_n(\omega)\} \in [e^{(-h-\epsilon)n}, e^{(-h+\epsilon)n}].
  \end{equation*}
Since $\log F_R(e,X_n)$ and $\log F_R(X_n, e)$ almost surely converge
to $c$, we can also assume (shrinking $E_n$ a little bit) that for
any $\omega \in E_n$ one has $F_R(e,X_n) \geq e^{(c-\epsilon)n}$ and
$F_R(X_n, e) \geq e^{(c-\epsilon)n}$. Let $\tilde E_n \subset \Gamma$
be the set of points $X_n(\omega)$ for $\omega\in E_n$. It has
cardinality at least $C e^{(h-\epsilon)n}$, it is contained in $B(e,
n)$ (since the steps of the random walk have length at most $1$ by
definition), and for any $x\in \tilde E_n$ one has $F_R(e,x)F_R(x,e)
\geq e^{2(c-\epsilon)n}$. Therefore,
  \begin{equation*}
  \sum_{x\in B(e,n)} F_R(e,x)F_R(x,e) \geq \sum_{x\in \tilde E_n} F_R(e,x)F_R(x,e)
  \geq \Card \tilde E_n \cdot e^{2(c-\epsilon)n}
  \geq C e^{(h-\epsilon)n} e^{2(c-\epsilon)n}.
  \end{equation*}
By Lemma~\ref{lem_GRbounded}, the sum $\sum_{x\in \Sbb_k}
F_R(e,x)F_R(x,e)$ is uniformly bounded (since
$F_R(x,y)=G_R(x,y)/G_R(e,e)$). Therefore, $\sum_{x\in B(e,n)}
F_R(e,x)F_R(x,e)  \leq C n$. We deduce that $(h-\epsilon) +
2(c-\epsilon)$ is nonpositive. Finally, letting $\epsilon$ tend to
$0$, we get $c\leq - h/2$. Since entropy is nonzero in non-amenable
groups (see for instance \cite[Proposition 2.35]{furman}), we get
$c<0$ as desired.

Let us now prove the estimates of the lemma. The first item (positive
escape rate) is classical and follows from Kingman's theorem for the
existence of the escape rate, and from the inequality $h\leq v \zeta$
for its positivity (where $\zeta$ is the exponential growth rate of
the cardinality of balls), see \cite[Proposition~2.32]{furman}. For
the second item, consider a typical trajectory $X_k$ of the random
walk. Since $\log F_R(e,X_k) \sim ck$ with $c<0$, we deduce that for
large enough $k$ one has $F_R(e,X_k) \leq e^{ck/2}$. Since $G_R(x,y)
= F_R(x,y) G_R(e,e)$, the exponential decay of $G_R(e,X_k)$ follows.
The decay of $G_R(X_k, e)$ is handled in the same way. Finally,
consider two independent trajectories $X_k$ and $Y_k$ of the random
walk. Define (for $k>0$) $X_{-k}=Y_k$. By symmetry of $\nu$,
$(X_k)_{k\in \Z}$ is a typical trajectory for the bilateral random
walk. Applying Theorem~\ref{thm_kingman}, we deduce that $\log
F_R(X_{-m}, X_n) \sim (m+n)c$, i.e., $\log F_R(Y_m, X_n) \sim (m+n)c$
when $m+n$ tends to infinity. This is the desired exponential decay.
\end{proof}

\bibliography{biblio}
\bibliographystyle{amsalpha}
\end{document}